\documentclass[11pt]{amsart}
\usepackage{amsfonts}
\usepackage{amsmath}
\usepackage{amssymb}
\usepackage{multicol}
\usepackage{stmaryrd}
\usepackage{cite}
\usepackage{epsfig}
\usepackage{color}
\usepackage{graphics}
\usepackage{graphicx}
\usepackage{multicol,graphics}
\newcommand\bes{\begin{eqnarray}}
\newcommand\ees{\end{eqnarray}}
\newcommand\R{\mathbb R}
\newtheorem{theorem}{Theorem}[section]
\newtheorem{lemma}[theorem]{Lemma}
\newtheorem{corollary}[theorem]{Corollary}

\newtheorem{remark}[theorem]{Remark}
\newtheorem{proposition}[theorem]{Proposition}
\numberwithin{equation}{section}
\allowdisplaybreaks

\begin{document}

\title[A nonlocal diffusion model
with free boundaries]{The dynamics of a Fisher-KPP nonlocal diffusion model with free boundaries}
\author[J.F. Cao, Y. Du, F. Li and W.T. Li]{Jia-Feng Cao$^\dag$, Yihong Du$^\ddag$, Fang Li$^\S$ and Wan-Tong Li$^\dag$}
\thanks{\hspace{-.6cm} $^\dag$ School of Mathematics and Statistics, Lanzhou University,
Lanzhou, Gansu, 730000,\\
\mbox{\hspace{.2cm} People's Republic of China.}
\\
$^\ddag$ School of Science and Technology, University of New England, Armidale, NSW 2351, Australia.
\\
$^\S$ ({\sf Corresponding author}) School of Mathematics,
Sun Yat-sen University,   Guangzhou, Guangdong, 510275,\\
\mbox{\hspace{.2cm}  People's Republic of China}}

\date{\today}

\maketitle

\begin{abstract}
We introduce and study a class of free boundary models with ``nonlocal diffusion'', which are natural extensions of
the free boundary models in \cite{DYH2010} and elsewhere, where ``local diffusion'' is used to describe the population dispersal, with the free boundary representing the spreading front of the species.
We show that this nonlocal problem has a unique solution defined for all time, and then examine its long-time dynamical behavior when the growth function is of Fisher-KPP type. We prove that a spreading-vanishing dichotomy holds, though for the spreading-vanishing criteria significant differences arise from the well known local diffusion model in \cite{DYH2010}.

\bigskip

\textbf{Keywords}: Nonlocal diffusion; Free
boundary;  Existence-uniqueness; Spreading-vanishing
dichotomy

\textbf{AMS Subject Classification (2000)}: 35K57,
35R20, 92D25

\end{abstract}

\section{Introduction}
\noindent

In this paper we introduce and study a free boundary model with ``nonlocal diffusion'' (to be described in detail later), which may be viewed as an extension of the following
free boundary model with ``local diffusion'':
\begin{equation}
\left\{
\begin{aligned}
&u_t-du_{xx}=f(u),& &t>0,~g(t)<x<h(t),\\
&u(t,g(t))=u(t,h(t))=0,& &t>0,\\
&g'(t)=-\mu u_x(t,g(t)),\; h'(t)=-\mu u_x(t,h(t)),& &t>0,\\
&g(0)=g_0,\; h(0)=h_0,~u(0,x)=u_0(x),& &g_0\le x\le h_0,
\end{aligned}
\right.
\label{1}
\end{equation}
where $f$ is a $C^1$ function satisfying $f(0)=0$,
$\mu>0$ and $g_0<h_0$ are  constants, and $u_0$ is a $C^2$ function which is positive in $(g_0, h_0)$ and
vanishes at $x=g_0$ and $x=h_0$. For logistic type of $f(u)$, \eqref{1} was first studied in \cite{DYH2010}, as a model for the spreading of a new or invasive species with population density $u(t,x)$, whose population range $(g(t), h(t))$ expands through its boundaries $x=g(t)$ and $x=h(t)$ according to the Stefan conditions
 $g'(t)=-\mu u_x(t, g(t)),\; h'(t)=-\mu u_x(t,h(t))$.
A deduction of these conditions based on some ecological assumptions can be found in \cite{Bunting}.

It was shown in\cite{DYH2010} that problem (\ref{1})
admits a unique solution $(u(t,x), g(t), h(t))$ defined for all $t>0$. Its long-time dynamical behavior is characterized by a ``spreading-vanishing dichotomy'':
Either $(g(t), h(t))$ is contained in a bounded set of $\mathbb R$ for all $t>0$ and $u(t,x)\to 0$ uniformly as $t\to\infty$ (called the vanishing case),
or $(g(t), h(t))$ expands to $\mathbb R$ and $u(t,x)$ converges to the unique positive steady state of the ODE $v'=f(v)$ locally uniformly in $x\in\mathbb R$ as $t\to\infty$ (the spreading case). Moreover, when spreading occurs,
\[
\lim_{t\to\infty} \frac{-g(t)}{t}=\lim_{t\to\infty}\frac{h(t)}{t}=k_0>0,
\]
and $k_0$ is uniquely determined by a traveling wave equation associated to \eqref{1}.

These results have been extended to cases with more general $f(u)$ in \cite{DYH2015, K-Yamada}, and more accurate estimates for $g(t), h(t)$ and $u(t,x)$ for the spreading case have been obtained  in \cite{DMZ2014}.  Among the many further extensions, we only mention
the extension to various heterogeneous environments in \cite{DYH2013JFA, DyhLx2015, LLS16-1, LLS16-2, DDL-1, DDL-2, WangMX2016JFA}, and extensions to certain Lotka-Volterra two-species systems and epidemic models
 in \cite{DYH2014, DWZ, GJS2012,wangmx1} and \cite{CaoJFZAMP2017, GeJing2015, LZ2017} (see also the references therein).

Problem \eqref{1} is closely related to the following associated Cauchy problem
\begin{equation}
\left\{
\begin{aligned}
&U_t-dU_{xx}=f(U),& &t>0,~x\in\mathbb R,\\
&U(0,x)=U_0(x),& & x\in \mathbb R,
\end{aligned}
\right.
\label{Cauchy-local}
\end{equation}
Indeed, it follows from \cite{DYH2012} that, if the initial functions are the same, i.e., $u_0=U_0$, then the unique solution $(u, g, h)$ of \eqref{1} and the unique solution $U$ of \eqref{Cauchy-local} are related in the following way: For any fixed $T>0$,
as $\mu\to\infty$, $(g(t), h(t))\to \mathbb R$ and $u(t,x)\to U(t,x)$ locally uniformly in $(t,x)\in(0, T]\times \mathbb R$. Thus \eqref{Cauchy-local} may be viewed as the limiting problem of \eqref{1} (as $\mu\to\infty$).

Problem \eqref{Cauchy-local} with $U_0$ a nonnegative function having nonempty compact support has long been used to describe the spreading of  a new or invasive species; see, for example, classical works \cite{Aronson1978, Fisher1937, KPP}. In both \eqref{1} and \eqref{Cauchy-local}, the dispersal of the species is described by the diffusion term
$d u_{xx}$, widely known as a ``local diffusion'' operator. It has been increasingly recognized that the dispersal of many species is better described by ``nonlocal diffusion'' rather than local diffusion as used in \eqref{1} and \eqref{Cauchy-local} (see, e.g., \cite{Nathan12}). An extensively used nonlocal diffusion operator
to replace the local diffusion term $d u_{xx}$ is given by
\[
d(J * u-u)(t,x):=d\left[\int_{\mathbb R} J(x-y)u(t,y)dy-u(t,x)\right],
\]
where $J:\R\setminus\{0\}\to\R$ is a continuous nonnegative even function satisfying
\[
2\int_0^\infty J(x)dx=1.
\]
The corresponding nonlocal version of \eqref{Cauchy-local} is thus given by
\begin{equation}
\left\{
\begin{aligned}
&u_t=d\left[\int_{\R}J(x-y)u(y,t)dy-u(x,t)\right]+f(u), && t>0, \; x\in\R,\\
&u(x,0)=u_0(x), && x\in\R.
\end{aligned}
\right.
\label{113}
\end{equation}

Problem \eqref{113} and its many variations have been extensively studied in recent years; see, for example,
\cite{Bates1997, BerestyckiJMB2016, ChenXF1997, Coville2007Proceedings, Hutson2003JMB,Kao2010DCDS, Rawal2015DCDS,Shen-XiexiaoxiaJDE2015, SunYJ2010NARWA,
SunYJ2011JDE, ZhangAiJun2010JDE, ZhangLiWang2012} and the references therein. In particular,
several classical results on \eqref{Cauchy-local} as a model for spreading of species have been successfully extended to \eqref{113} and its variations.

In contrast, the nonlocal version of \eqref{1} has not been considered so far. The main purpose of this paper is to propose a nonlocal version of \eqref{1} and then study its well-posedness and long-time dynamical behavior. We note that a solution $u(t,x)$ of \eqref{113}  is not differentiable in $x$ in general, and therefore
it is reasonable to expect that the Stefan conditions in \eqref{1} do not extend readily to the nonlocal diffusion case. Moreover, when \eqref{113} is restricted to a bounded interval $[a, b]$ for the space variable $x$,  it is well known that
for $t>0$, $u(t,x)>0$ for $x\in\{a,\, b\}$ and the values of $u(t,a)$ and $u(t, b)$ are implicitly determined by the equation itself. As we will see below, this will no longer be the case
for the corresponding free boundary problem at the free boundaries $x=g(t)$ and $x=h(t)$.

The nonlocal version of \eqref{1} we propose in this paper has the following form:
\begin{equation}
\left\{
\begin{aligned}
&u_t=d\int_{g(t)}^{h(t)}J(x-y)u(t,y)dy-du(t,x)+f(t,x,u),
& &t>0,~x\in(g(t),h(t)),\\
&u(t,g(t))=u(t,h(t))=0,& &t>0,\\
&h'(t)=\mu\int_{g(t)}^{h(t)}\int_{h(t)}^{+\infty}
J(x-y)u(t,x)dydx,& &t>0,\\
&g'(t)=-\mu\int_{g(t)}^{h(t)}\int_{-\infty}^{g(t)}
J(x-y)u(t,x)dydx,& &t>0,\\
&u(0,x)=u_0(x),~h(0)=-g(0)=h_0,& &x\in[-h_0,h_0],
\end{aligned}
\right.
\label{101}
\end{equation}
where $x=g(t)$ and $x=h(t)$ are the moving boundaries
to be determined together with $u(t,x)$, which is always assumed to be identically 0 for $x\in \R\setminus [g(t), h(t)]$; $d$ and $\mu$
 are positive constants. The initial
function $u_0(x)$ satisfies
\begin{equation}
u_0(x)\in C([-h_0,h_0]),~u_0(-h_0)=u_0(h_0)=0
~\text{ and }~u_0(x)>0~\text{ in }~(-h_0,h_0),
\label{102}
\end{equation}
with  $[-h_0,h_0]$ representing the initial population range of the species. We assume that the kernel function
$J: \mathbb{R}\rightarrow\mathbb{R}$ is continuous and nonnegative,
and has the properties
\begin{description}
\item[(J)]$J(0)>0,~\int_{\mathbb{R}}J(x)dx=1$, $J$  is symmetric, $\sup_{\R}J<\infty$.
\end{description}
The growth term $f: \mathbb{R}^+\times\mathbb{R}\times\mathbb{R}^+
\rightarrow\mathbb{R}$ is assumed to be continuous and satisfies
\begin{description}
\item[(f1)] $f(t,x,0)\equiv 0$ and $f(t,x,u)$
is locally Lipschitz in $u\in\R^+$, i.e., for any $L>0$, there

 \hspace{0.1cm} exists a constant $K=K(L)>0$
such that
\[\hspace{0.8cm}
\mbox{$\left|f(t,x,u_1)-f(t,x,u_2)\right|\le K|u_1-u_2|$ for $u_1, u_2\in [0, L]$ and $(t,x)\in \R^+\times \R$;}
\]
\end{description}
\begin{description}
\item[(f2)] There exists
$K_0>0$ such that $f(t,x,u)<0$ for $u\ge K_0$ and $(t,x)\in \R^+\times \R$.
\end{description}

\medskip

Let us now explain the meaning of the free boundary conditions in \eqref{101}. According to the nonlocal dispersal rule governed
by the dispersal kernel
$J(x-y)$, the total population mass moved out of the range $[g(t), h(t)]$ at time $t$ through its right boundary $x=h(t)$ per unit time is given by
\[
\int_{g(t)}^{h(t)}\int_{h(t)}^\infty J(x-y)u(t,x)dydx.
\]
As we assume that $u(t,x)=0$ for $x\not\in [g(t), h(t)]$, this quantity of mass is lost in the spreading process of the species.
We may call this quantity the outward flux at $x=h(t)$ and denote it by $J_h(t)$.
Similarly we can define the outward flux at $x=g(t)$ by
\[
J_g(t):=\int_{g(t)}^{h(t)}\int_{-\infty}^{g(t)}J(x-y)u(t,x)dydx.
\]
Then the free boundary conditions in \eqref{101} can be interpreted as assuming that the expanding rate of the front
is proportional to the outward flux, by a factor $\mu>0$:
\[
g'(t)=-\mu J_g(t),\; h'(t)=\mu J_h(t).
\]

So the population range expands at a rate proportional to the total population across the front. For a plant species, seeds  carried across the range boundary may fail to establish due to numerous reasons, such as isolation from other members of the species causing poor or no pollination, or causing overwelming attacks from enemy species. However, some of those not very far from the range boundary may survive, which results in the expansion of the population range. It is reasonable to assume that this survival rate is roughly a constant for a given species, and therefore the population range should expand at a rate proportional to the total population carried across the front. For an animal species, a simillar consideration can be applied. We would like to emphasize  that for most species, its living environment involves many factors, not only the resources such as food or nutrient supplies. For example,  complex interactions of the concerned species with many other species in the same spatial
habitat constantly occur, yet it is impossible to include all of them (even the majority of them) into a manageable model, and best treat them, or rather their combined effects, as part of the environment of the concerned species.

\bigskip

We note that in \eqref{101}, the  growth function $f$ is more general than that in \eqref{1} where $f$ is independent of $(t,x)$. Our results in this paper on global existence and uniqueness of the solution to \eqref{101} will be proved under the conditions {\bf (J), (f1)} and {\bf (f2)} only. More precisely, we have the following result:
\begin{theorem}\label{thm1}
Under the conditions {\bf (J), (f1)} and {\bf (f2)}, for every $u_0$ satisfying \eqref{102}, problem
\eqref{101} has a unique solution $(u,g,h)$ defined for all $t>0$.
\end{theorem}

For the long-time dynamical behavior of \eqref{101}, for simplicity, we will add more restrictions on  $f$. More precisely, we will require additionally

\begin{description}
\item[(f3)] $f=f(u)$ is independent of $(t,x)$, and
$\frac{f(u)}{u}$ is strictly decreasing for $u\in \mathbb{R}^+$,
\item[(f4)]  $f'(0)$ exists and $f'(0)>0$.
\end{description}

A prototype of growth functions $f(u)$ satisfying {\bf (f1)-(f4)} is given by $f(u)=au-bu^2$ with $a$ and $ b$ positive constants.
We have the following conclusions:

\begin{theorem}[Spreading-vanishing dichotomy]\label{thm2}
Suppose {\bf (J)} and {\bf (f1)-(f4)} hold, and $u_0$ satisfies \eqref{102}. Let $(u,g,h)$ be the unique solution of problem
\eqref{101}. Then one of the following alternatives must happen for \eqref{101}:
\begin{itemize}
\item[(i)] \underline{\rm Spreading:} $\lim_{t\to+\infty} (g(t), h(t))=\mathbb R$ and $\lim_{t
\rightarrow+\infty}u(t,x)=v_0$ locally uniformly  in
$\mathbb{R}$, where $v_0$ is the unique positive zero of $f(u)$,
\item[(ii)] \underline{\rm Vanishing:} $\lim_{t\to+\infty} (g(t), h(t))=(g_\infty, h_\infty)$ is a finite interval and $\lim_{t
\rightarrow+\infty}u(t,x)=0$ uniformly for $x\in [g(t),h(t)]$.
\end{itemize}
\end{theorem}

\begin{theorem}[Spreading-vanishing criteria]\label{thm3}
Under the conditions of Theorem \ref{thm2}, if $f'(0)\geq d$, then spreading always happens. If $f'(0)\in (0, d)$, then there exists a unique $\ell^*>0$ such that
spreading always happens if $h_0\geq \ell^*/2$; and for $h_0\in (0, \ell^*/2)$, there exists a unique $\mu^*>0$ so that spreading happens exactly when $\mu>\mu^*$.
\end{theorem}

As we will see in Section 3.2, $\ell^*$ depends only on $f'(0), \; d$ and the kernel function $J$. On the other hand, $\mu^*$ depends also on $u_0$.

\begin{remark}
We note that for the corresponding local diffusion model in \cite{DYH2010}, no matter how small is the diffusion coefficient $d$ in $d u_{xx}$  relative to $f'(0)$, vanishing can always happen if $h_0$ and $\mu$ are both sufficiently small. However, for \eqref{101}, Theorem \ref{thm3} indicates that
when $d\leq f'(0)$, spreading always happens no mater how small $h_0$, $\mu$ and $u_0$ are. This is one of the most striking differences between
the local and nonlocal diffusion models discovered in this paper.
\end{remark}

\begin{remark}
In a forthcoming paper, we will study the spreading speed of \eqref{101} when spreading happens.
\end{remark}

\begin{remark}
After this paper is completed, we learned of the preprint \cite{CQW2018}, where \eqref{101} with $f\equiv 0$
is studied, and some high space dimension cases are also considered. Interesting asymptotic behavior of the solution 
to \eqref{101} with $f\equiv 0$ is established in \cite{CQW2018}, which naturally is very different from our case here where a Fisher-KPP growth term $f$ is used in \eqref{101}.
\end{remark}

\bigskip

The rest of the paper is organized as follows. In
Section 2 we prove that (\ref{101}) has a unique solution defined for all $t>0$.
Although this is based on the contraction mapping theorem, our approach is very different from those used in similar free boundary problems with
``local diffusion''. For example, we introduce and use  a parameterized ODE problem in the proof of Lemma 2.3, which is a key step
for our first contraction mapping argument to work; in the proof of Theorem 2.2, we make use of the fact that $h'(t)$ and $-g'(t)$ have a positive lower bound for small $t>0$, which is crucial for our proof of the second contraction mapping result.
In Section 3, we study the long-time dynamical behavior of the solution to \eqref{101} under the conditions {\bf (J)} and {\bf (f1)-(f4)}, where Theorems \ref{thm2} and \ref{thm3} will be proved. Here we make use of various comparison principles, and also properties of principal eigenvalues
of the associated nonlocal linear operators, and initial value problems over fixed spatial domains.

\section{Global existence and uniqueness for  (\ref{101})}
\noindent

In this section we prove the global existence and
uniqueness of the solution to problem (\ref{101}).
For convenience, we first introduce some notations. For given $h_0, T>0$ we define
\begin{align*}
&\mathbb H_{h_0, T}:=\Big\{h\in C([0,T])~:~h(0)=h_0,
\; \inf_{0\le t_1<t_2\le T}\frac{h(t_2)-h(t_1)}{t_2-t_1}>0\Big\},\\
&\mathbb G_{h_0, T}:=\Big\{g\in C([0,T])~:-g\in\mathbb{H}_{h_0, T}\Big\},\\
&C_0([-h_0, h_0]):=\Big\{u\in C([-h_0, h_0]): u(-h_0)=u(h_0)=0\Big\}.
\end{align*}
For $g\in \mathbb G_{h_0, T}$, $h\in\mathbb H_{h_0, T}$ and $u_0\in C_0([-h_0, h_0])$ nonnegative, we define
\begin{align*}
&\Omega=\Omega_{g, h}:=\left\{(t,x)\in\mathbb{R}^2: 0<t\leq T,~g(t)<x<h(t)\right\},\\
&\mathbb{X}=\mathbb X_{u_0,g,h}:=\Big\{\phi\in C(\overline\Omega_{g,h})~:~\phi\ge0~\text{in}
~\Omega_{g,h},~\phi(0,x)=u_0(x)~\text{for}~x\in [-h_0,h_0]~\\
& \hspace{6cm} \text{and} \;\; \phi(t,g(t))=
\phi(t,h(t))=0~\text{for }0\le t\le T\Big\}.
\end{align*}

The following theorem is the main result of this section.

\begin{theorem}
Suppose that {\rm \bf (J)} and {\rm \bf (f1)-(f2)} hold. Then for any given $h_0>0$ and $u_0(x)$
satisfying \eqref{102}, problem \eqref{101} admits a unique
 solution $(u(t,x), g(t), h(t))$ defined for all $t>0$. Moreover, for any $T>0$, $g\in \mathbb G_{h_0, T},\;
 h\in\mathbb H_{h_0, T}$ and $u\in \mathbb{X}_{u_0, g, h}$.
\label{Thm22}
\end{theorem}

The rest of this section is devoted to the proof of Theorem \ref{Thm22}.
The following Maximum Principle will be used frequently in our analysis to follow.
{\begin{lemma}[Maximum Principle]\label{lemma-MP}
Assume that {\rm \bf (J)}   holds, and $g\in \mathbb G_{h_0, T},\; h\in\mathbb H_{h_0, T}$ for some $h_0, T>0$.
Suppose that
 $u(t,x)$ as well as $u_t(t,x)$ are
continuous in $\overline\Omega_{g,h}$ and satisfies, for some $c\in L^\infty (\Omega_{g,h})$,
\begin{equation}\label{lemma-MP-u}
\left\{
\begin{aligned}
&u_t(t,x)\ge d\int_{g(t)}^{h(t)}J(x-y)u(t,y)dy-du +c(t,x)u, && t\in (0, T],\  x\in (g(t), h(t)),\\
& u(t, g(t)) \geq 0,\ u(t, h(t)) \geq 0, && t>0,\\
&u(0,x)\ge0,  && x\in [-h_0,  h_0].
\end{aligned}
\right.
\end{equation}
Then $u(t,x)\ge0$ for all $0\le t\le T$ and $x\in[g(t),h(t)]$.
Moreover, if $u(0,x)\not\equiv0$ in $[-h_0, h_0]$, then $u(t,x)>0$ in $\Omega_T$.
\end{lemma}}

\begin{proof} To simplify notations, for $s\in (0, T]$ we write
\[
\Omega_s:=\big\{(t,x)\in \Omega_{g,h}: \; t\leq s\big\}.
\]
Let $w(t,x)=e^{kt}u(t,x)$, where $k>0$ is a constant chosen large enough such that
\[
\mbox{$p(t,x):=k-d+ c(t,x )>0$ for all $(t,x)\in\Omega_T$.}
\]
 Then
\bes\label{pf-lm-eqn-w}
w_t&\ge& d\int_{g(t)}^{h(t)}J(x-y)w(t,y)dy+\left[k-d+ c(t,x )\right]w(t,x).
\ees

Denote  $p_0=\sup_{(t,x)\in\Omega_T}p(t,x)$ and $T_*=\min
\left\{T,~\frac{1}{2(d+p_0)}\right\}$. We are now in a position
to prove that $w\ge0$ in $\Omega_{T_*}$.
Suppose that
$$
w_{\inf}:=\inf_{(t,x)\in\Omega_{T_*}}w(t,x)<0.
$$
By (\ref{lemma-MP-u}), $w\geq 0$ on the parabolic boundary of $\Omega_{T}$, and hence there exists $(t_*, x_*)\in \Omega_{T_*}$ such that
$w_{\inf} = w (t_*, x_*)<0.$
We now define
\begin{equation*}
t_0 =t_0(x_*):=\left\{
\begin{aligned}
&t_g& & \mbox{ if $x_*\in(g(t_*),-h_0)$ and $g(t_g)=x_*$},\\
&0& &\mbox{ if $x_*\in[-h_0,h_0]$},\\
&t_h& &\mbox{ if $x_*\in(h_0,h(t_*))$ and $h(t_h)=x_*$}.
\end{aligned}
\right.
\end{equation*}
Clearly
\begin{equation*}
u(t_0,x_*)=\left\{
\begin{aligned}
&u(t_0, g(t_0))=0  & &\mbox{ if } x_*\in(g(t_*),-h_0),\\
&u_0(x_*)\geq 0  & & \mbox{ if }  x_*\in[-h_0,h_0],\\
&u(t_0, h(t_0))=0  & & \mbox{ if } x_*\in(h_0,h(t_*)).
\end{aligned}
\right.
\end{equation*}
Integrating (\ref{pf-lm-eqn-w}) from $t_0$ to $t_*$ we obtain
\begin{eqnarray*}
w(t_*,x_*)-w(t_0,x_*)  &\ge&  d\int_{t_0}^{t_*}\int_{g(t)}^{h(t)}
J(x_*-y)w(t,y)dydt+\int_{t_0}^{t_*}p(t,x_*)w(t,x_*)dt\\
&\ge& d\int_{t_0}^{t_*}\int_{g(t)}^{h(t)}J(x_*-y)w_{\inf}dydt
+\int_{t_0}^{t_*}p(t,x_*)w_{\inf}dt\\
&\ge& d\int_{t_0}^{t_*}\int_{\R}J(x_*-y)w_{\inf}dydt
+\int_{t_0}^{t_*}p_0w_{\inf}dt\\
 &= & \left(t_*-t_0 \right)(d+p_0)w_{\inf}.
\end{eqnarray*}
Since $w(t_0,x_*)=e^{kt_0}u(t_0,x_*)\ge0$, we deduce
$$
w_{\inf}\ge (t_*-t_0)(d+p_0)w_{\inf}\ge T_*(d+p_0)w_{\inf}\geq {1\over 2}w_{\inf},
$$
which is a contradiction to the assumption that $w_{\inf}<0$. Hence it follows that $w(t,x)\ge0$ and thus $u(t,x)\ge0$
for all $(t,x)\in\Omega_{T_*}$.

If $T_*=T$, then $u(t,x)\ge0$ in $\Omega_T$ follows directly; while if $T_*<T$, we may repeat this process with $u_0(x)$ replaced by $u(T_*, x)$ and
$(0, T]$ replaced by $(T_*, T]$. Clearly after repeating this process finitely many times, we will obtain $u(t,x)\ge0$ for all $(t,x)\in\Omega_T$.

Now assume that $u(0,x)\not\equiv0$ in $[-h_0, h_0]$. To complete the proof, it suffices to show that $w>0$ in $\Omega_T$.  Suppose that  there is a point $(t^*,x^*)\in
\Omega_T$ such that $w(t^*,x^*)=0$.

First, we claim that
\[
\mbox{$w(t^*, x)=0$ for $x\in (g(t^*), h(t^*))$.}
\]
 Otherwise,  there exists
$$
\tilde x \in [g(t^*), h(t^*)] \cap \partial \big\{x\in (g(t^*), h(t^*))  : \ w(t^*, x) >0   \big\}.
$$
Then at $(t^*,\tilde  x)$, $u(t^*, \tilde x)=0$ and by (\ref{pf-lm-eqn-w}), we get
$$
0\ge w_t(t^*, \tilde x)\ge d\int_{g(t^*)}^{h(t^*)}J(\tilde x-y)  w(t^*,y) dy > 0,
$$
due to assumption \textbf{(J)}.
This is impossible. Thus $w(t^*, x)=0$ for $x\in (g(t^*), h(t^*))$.
Hence, by (\ref{pf-lm-eqn-w}), for $x\in [-h_0, h_0]\subset (g(t^*), h(t^*))$,
\begin{eqnarray*}
-w(0,x)&=&w(t^*, x ) -w(0,x)\\
 &\geq & d\int_0^{t^*}\int_{g(t)}^{h(t)}J(x-y)w(t,y)dy dt + \int_0^{t^*} p(t,x )w(t,x)dt\geq 0.
\end{eqnarray*}
This implies that $u(0,x) \equiv0$ in $[-h_0, h_0]$, which is a contradiction.
\end{proof}

The following result will play a crucial role in the proof of Theorem 2.1.

\begin{lemma}
Suppose that {\rm \bf (J)} and {\rm \bf (f1)-(f2)} hold, $h_0>0$ and $u_0(x)$
satisfies \eqref{102}. Then for any $T>0$ and $(g, h)\in\mathbb G_{h_0,T}\times \mathbb H_{h_0, T}$,
 the following problem
\begin{equation}
\left\{
\begin{aligned}
&v_t=d\int_{g(t)}^{h(t)}J(x-y)v(t,y)dy-dv+f(t,x,v),
& &0<t< T,~x\in(g(t),h(t)),\\
&v(t,h(t))=v(t,g(t))=0,& &0<t< T,\\
&v(0,x)=u_0(x),& &x\in[-h_0,h_0]
\end{aligned}
\right.
\label{201}
\end{equation}
admits a unique  solution, denoted by $V_{g,h}(t,x)$.
 Moreover $V_{g,h}$ satisfies
\begin{equation}
0<V_{g,h}(t,x)\le \max\left\{\max_{-h_0\le x\le h_0}u_0(x),
~K_0\right\}~\mbox{ for } 0<t< T,~x\in(g(t),h(t)),
\label{v-bound}
\end{equation}
where $K_0$ is defined in the assumption {\rm\bf (f2)}.
\label{Lemma202}
\end{lemma}

\begin{proof} We break the proof into three steps.

\medskip

\noindent {\bf Step 1:} {\it A parametrized ODE problem.}

For given $x\in[g(T),h(T)]$, define
\begin{equation}
\tilde u_0(x)=\left\{
\begin{aligned}
&0,& &x\not\in[-h_0,h_0],\\
&u_0(x),& &x\in[-h_0,h_0]
\end{aligned}
\right.
~~~\text{and }~~~t_x=\left\{
\begin{aligned}
&t_{x,g}& &\mbox{ if  $x\in[g(T),-h_0)$ and $g(t_{x,g})=x$},\\
&0& &\mbox{ if } x\in[-h_0,h_0],\\
&t_{x,h}& &\mbox{ if $x\in(h_0,h(T)]$ and $h(t_{x,h})=x$}.
\end{aligned}
\right.
\label{definition-of-t_x}
\end{equation}
Clearly $t_x=T$ for $x=g(T)$ and $x=h(T)$, and $t_x<T$ for $x\in (g(T), h(T))$.
 For any given $\phi\in
\mathbb{X}_{u_0, g,h}$, consider the following ODE initial value problem (with parameter $x$):
\begin{equation}\label{202}
\left\{
\begin{aligned}
 &v_t=d\int_{g(t)}^{h(t)}J(x-y)\phi(t,y)dy-dv(t,x)+\tilde f(t,x,v),& &t_x<t\le T,\\
&v(t_x,x)=\tilde u_0(x),& &x\in(g(T),h(T)),
\end{aligned}
\right.
\end{equation}
where $\tilde f(t,x,v)=0$ for $v<0$, and $\tilde f(t,x,v)=f(t,x,v)$ for $v\geq 0$. Clearly $\tilde f$ also satisfies  {\rm \bf (f1)-(f2)}.
Denote
$$
F(t,x,v)=d\int_{g(t)}^{h(t)}J(x-y)\phi(t,y)dy-dv(t,x)
+\tilde f(t,x,v).
$$
Thanks to the assumption \textbf{(f1)}, for any $v_1,v_2
\in (-\infty, L]$, we have
$$
\Big|F(t,x,v_1)-F(t,x,v_2)\Big|\le\Big|\tilde f(t,x,v_1)-\tilde f(t,x,v_2)\Big|
+d\Big|v_1-v_2\Big|\le K_1\Big|v_1-v_2\Big|,
$$
where
\[
  L:=1+\max\Big\{\|\phi\|_{C(\overline\Omega_T)},  K_0\Big\},\; K_1:=d+K(L).
\]
In other words, the function $F(t,x,v)$ is Lipschitz continuous in $v$ for $v\in (-\infty, L]$
 with Lipschitz
constant $K_1$, uniformly for $t\in [0, T]$ and $x\in (g(T), h(T))$. Additionally, $F(t,x,v)$ is continuous
in all its variables in this range.
Hence it follows from the Fundamental Theorem of
ODEs  that, for every fixed $x\in (g(T), h(T))$, problem (\ref{202}) admits a
unique solution, denoted by $V_{\phi}(t,x)$ defined in some  interval $[t_x,T_x)$ of $t$.

We claim that $t\to V_\phi(t,x)$ can be uniquely extended to $[t_x, T]$. Clearly it suffices to show that if $V_\phi(t,x)$ is uniquely defined for $t\in [t_x, \tilde T]$ with $\tilde T\in (t_x, T]$, then
\begin{equation}\label{V-bd}
0\leq V_\phi(t,x)< L \mbox{ for } t\in (t_x, \tilde T].
\end{equation}

We first show that $V_\phi(t,x)<L$ for $ t\in (t_x, \tilde T]$. Arguing indirectly we assume that this inequality does not hold,
and hence, in view of $V_\phi(t_x,x)=\tilde u_0(x)\leq \|\phi\|_{C(\overline\Omega_T)} <L$, there exists some $t^*\in (t_x, \tilde T]$ such that
$V_\phi(t,x)<L$ for $t\in (t_x, t^*)$ and $V_\phi (t^*,x)=L$. It follows that $(V_\phi)_t(t^*,x)\geq 0$ and $\tilde f(t^*,x, V_\phi(t^*,x))\leq 0$
(due to $L>K_0$). We thus obtain from the differential equation satisfied by $V_\phi(t,x)$ that
\[
dL=dV_\phi(t^*,x)\leq d\int_{g(t^*)}^{h(t^*)}J(x-y)\phi(t^*,y)dy\leq d\|\phi\|_{C(\overline\Omega_T)}\leq d(L-1).
\]
It follows that $L\leq L-1$. This contradiction proves our claim.

We now prove the first inequality in \eqref{V-bd}. Since
\[
\tilde f(t,x,v)=\tilde f(t,x,v)-\tilde f(t,x,0)\geq -K(L) |v| \mbox{ for } v\in (-\infty, L],
\]
we have
\[
(V_\phi)_t\geq -K_1{\rm sgn}(V_\phi) V_\phi +d\int_{g(t)}^{h(t)}J(x-y)\phi(t,y)dy\geq -K_1 {\rm sgn}(V_\phi)V_\phi \mbox{ for } t\in [t_x, \tilde t].
\]
Since $V_\phi(t_x,x)=\tilde u_0(x)\geq 0$, the above inequality immediately gives $V_\phi(t,x)\geq 0$ for $t\in [t_x, \tilde T]$.
 We have thus proved \eqref{V-bd}, and therefore the solution $V_\phi(t,x)$ of  \eqref{202} is uniquely defined for $t\in [t_x, T]$.

\medskip

\noindent {\bf Step 2:} {\it A fixed point problem.}

Let us note that $V_{\phi}(0,x)=u_0(x)$ for $x\in [-h_0, h_0]$,
and $V_{\phi}(t,x)=0$ for $t\in [0, T)$ and $x\in\partial(g(t), h(t))=\big\{ g(t), h(t)\big\}$. Moreover, by the continuous dependence of the unique solution on the initial value and on the parameters in the equation, we also see that $V_\phi(t,x)$ is continuous in $(t,x)\in\overline \Omega_s$
for any $s\in (0, T)$, and hence $V_\phi|_{\overline\Omega_s}\in\mathbb X_s$, where for convenience of notation, we define
for any $s\in(0,T]$,
\[
\Omega_s:=\Big\{(t,x)\in\Omega_{g,h}: t\leq s\Big\},\; \mathbb X_s:=\Big\{\psi|_{\overline\Omega_s}:\;\psi\in\mathbb X_{u_0,g,h}\Big\}.
\]
We then
define the mapping $\Gamma_s: \mathbb X_s\to \mathbb X_s$ by
\[
\Gamma_s \psi=V_{\tilde \psi}|_{\overline\Omega_s},
\]
where $\tilde\psi\in \mathbb X_T$ is an extension of $\psi$ from $\mathbb X_s$ to $\mathbb X_T$. From the definition of
$V_{\tilde\psi}$ we easily see that $\Gamma_s\psi$ does not depend on the particular form of the extension from $\psi$ to $\tilde \psi$.
We also see that if $\Gamma_s \psi=\psi$ then $\psi(t,x)$ solves \eqref{201} for $t\in (0, s]$, and vice versa.

The main task in this step is to show that for sufficiently small $s>0$, $\Gamma_s$ has a unique fixed point in $\mathbb X_s$.
We prove this conclusion by the contraction mapping theorem; namely we prove that for such $s$, $\Gamma_s$ is a contraction mapping on a closed subset of $\mathbb X_s$, and any fixed point of $\Gamma_s$ in $\mathbb X_s$ lies in this closed subset.

Firstly we note that
 $\mathbb{X}_s$ is a complete metric space
with the metric
$$
d(\phi_1,\phi_2)=\|\phi_1-\phi_2\|_{C(\overline\Omega_s)}.
$$
Fix $M> \max\big\{4\|u_0\|_\infty, K_0\big\}$ and define
\[
\mathbb X_s^M:=\big\{ \phi\in \mathbb X_s: \; \|\phi\|_{C(\overline \Omega_s)}\leq M\big\}.
\]
Clearly $\mathbb X_s^M$ is a closed subset of $\mathbb X_s$. We show next that there exists $\delta>0$ small depending on $M$ such that
for every $s\in (0, \delta]$, $\Gamma_s$ maps $\mathbb X_s^M$ into itself, and is a contraction mapping.

Let $\phi\in \mathbb X_s^M$ and denote $v=\Gamma_s \phi$. Then $v$ solves  \eqref{202}  with $T$ replaced by $s$. It follows that \eqref{V-bd} holds
with \textcolor{red}{$\tilde T$} replaced by $s$ and $V_\phi$ replaced by $v$. We prove that for all small $s>0$,
\[
v(t,x)\leq M \mbox{ for } t\in [t_x, s],\; x\in (g(s), h(s)),
\]
which is equivalent to $\|v\|_{C(\overline \Omega_s)}\leq M$.

Let us observe that due to {\rm\bf (f1)-(f2)}, there exists $K_*>0$ such that
\[
f(t,x,u)\leq K_* u \mbox{ for all } u\in [0,\infty).
\]
Now from \eqref{202} we obtain, for $ t\in [t_x, s]$ and $ x\in (g(s), h(s))$,
\[
v_t\leq d\int_{g(t)}^{h(t)}J(x-y)\phi(t,y)dy+K_*v\leq d\|\phi\|_{C(\overline \Omega_s)}+K_*v.
\]
It follows that, for such $t$ and $x$,
\[
e^{-K_*t} v(t, x)- e^{-K_*t_x}  v(t_x,x)\leq d\int_{t_x}^t e^{-K_*\tau}d\tau \|\phi\|_{C(\overline \Omega_s)},
\]
and
\[
v(t,x)\leq \|u_0\|_\infty e^{K_*t}+d(t-t_x) e^{K_*t}  \|\phi\|_{C(\overline \Omega_s)}\leq \|u_0\|_\infty e^{K_*s}+ds e^{K_* s } M.
\]
If $\delta_1>0$ is small enough such that
\[
d\delta_1  e^{K_* \delta_1} \leq \frac 14,\; e^{K_*\delta_1}\leq 2,
\]
then for $s\in (0,\delta_1]$ we have
\[
v(t,x)\leq \frac 14(8 \|u_0\|_\infty+M)\leq M \mbox{ in } \Omega_s.
\]
Thus $v=\Gamma_s\phi\in \mathbb X_s^M$, as we wanted. Let us note from the above choice of $\delta_1$ that it only depends on
$d$ and $K_*$.

Next we show that by shrinking $\delta_1$ if necessary, $\Gamma_s$ is a contraction mapping on $\mathbb X_s^M$ when $s\in (0, \delta_1]$.
So let $\phi_1,\phi_2\in\mathbb{X}_s^M$, and denote
$V_i=\Gamma_s\phi_i$, $i=1,2$. Then $w=V_1-V_2$ satisfies
\begin{equation}
\left\{
\begin{aligned}
&w_t+c_1(t,x)w=d\int_{g(t)}^{h(t)}J(x-y)\left(\phi_1-\phi_2\right)(t,y)dy,& &t_x<t\le s,~x\in(g(t),h(t)),\\
&w(t_x,x)=0,& &x\in (g(t), h(t)),
\end{aligned}
\right.
\label{203}
\end{equation}
where
$$
c_1(t,x):=d-\frac{f(t,x,V_1)-f(t,x,V_2)}{V_1-V_2} \mbox{ and hence } \|c_1\|_\infty\leq K_1(M):=d+K(M).
$$
It follows that, for $t_x<t\le s$ and $x\in  (g(t),h(t))$,
\begin{align*}
w(t,x)=de^{-\int_{t_x}^tc_1(\tau,x)d\tau}
\int_{t_x}^te^{\int_{t_x}^\xi c_1(\tau,x)d\tau}\int_{g(\xi)}^{h(\xi)}J(x-y)\left(
\phi_1-\phi_2\right)(\xi,y)dyd\xi.
\end{align*}
We thus deduce, for such $t$ and $x$,
\begin{align*}
\Big|w(t,x)\Big|&\le de^{K_1(M)(t-t_x)}\|\phi_1-\phi_2\|_{C(
\overline\Omega_s)}\int_{t_x}^te^{K_1(M)(\xi-t_x)}d\xi\\
&\le de^{K_1(M)s}\|\phi_1-\phi_2\|_{C(\overline \Omega_s)}\cdot
(t-t_x)e^{K_1(M)(t-t_x)}\\
&\le s d\, e^{2K_1(M)s}\|\phi_1-\phi_2\|_{C(\overline\Omega_s)}.
\end{align*}
Hence
$$
\|\Gamma_s \phi_1-\Gamma_s\phi_2\|_{C(\overline\Omega_s)} =\|w\|_{C(\overline\Omega_s)}\le\frac 12\|\phi_1-\phi_2\|_{C(\overline
\Omega_s)}~\text{ for }~s\in (0, \delta],
$$
provided that $\delta\in (0,\delta_1]$ satisfies
\[
\delta d\, e^{2K_1(M)\delta}\leq \frac 12.
\]
For such $s$ we may now apply the Contraction Mapping Theorem to conclude that $\Gamma_s$ has a unique fixed point $V$ in $\mathbb X_s^M$. It follows that $v=V$ solves \eqref{201} for $0<t\leq s$.

If we can show that any solution $v$ of \eqref{201} satisfies $0\leq v\leq M$ in $\Omega_s$ then $v$ must coincides with the unique fixed point $V$ of $\Gamma_s$ in $\mathbb X_s^M$. We next prove such an estimate for $v$.
We note that $v\geq 0$ already follows from \eqref{V-bd}. So we only need to prove $v\leq M$. We actually prove the following stronger inequality
\begin{equation}\label{v-upper-bd}
v(t,x)\leq M_0:=\max\big\{\|u_0\|_\infty,\; K_0\big\}<M \mbox{ for } t\in [t_x, s],\; x\in (g(s), h(s)).
\end{equation}
It suffices to show that the above inequality holds with $M_0$ replaced by $M_0+\epsilon$ for any given $\epsilon>0$.
We argue by contradiction. Suppose this is not true. Then due to
$v(t_x, x)=\tilde u_0(x)\leq \|u_0\|_\infty<M_\epsilon:=M_0+\epsilon$, there exists some $x^*\in (g(s), h(s))$ and $t^*\in (t_x, s]$ such that
\[
v(t^*, x^*)=M_\epsilon \mbox{ and }
0\leq v(t,x)<M_\epsilon \mbox{ for } t\in [t_x, t^*),\; x\in (g(s), h(s)).
\]
It follows that $v_t(t^*, x^*)\geq 0$ and $f(t^*, x^*, v(t^*, x^*))\leq 0$.
Hence from \eqref{201} we obtain
\[
0\leq v_t(t^*,x^*)\leq d\int_{g(t^*)}^{h(t^*)}J(x^*-y) v(t^*,y)dy-d v(t^*, x^*).
\]
Since $v(t^*,g(t^*))=v(t^*, h(t^*))=0$, for $y\in (g(t^*), h(t^*))$ but close to the boundary of this interval, $v(t^*,y)<M_\epsilon$.
It follows that
\[
dM_\epsilon=dv(t^*,x^*)\leq d\int_{g(t^*)}^{h(t^*)}J(x^*-y) v(t^*,y)dy<dM_\epsilon\int_{g(t^*)}^{h(t^*)}J(x^*-y)dy\leq dM_\epsilon.
\]
This contradiction proves \eqref{v-upper-bd}.
 Thus $v$ satisfies the wanted inequality and hence coincides with the unique fixed point
of $\Gamma_s$ in $\mathbb X_s^M$. We have now proved the fact that for every $s\in (0,\delta]$, $\Gamma_s$ has a unique fixed point
in $\mathbb X_s$.

\medskip

\noindent
{\bf Step 3:} {\it Extension and completion of the proof.}

From Step 2 we know that \eqref{201} has a unique solution defined for $t\in [0, s]$ with $s\in (0, \delta]$. Applying Step 2
to \eqref{201} but with the initial time $t=0$ replaced by $t=s$ we see that the unique solution can be extended to a slightly bigger interval
of $t$. Moreover, by \eqref{v-upper-bd} and the definition of $\delta$ in Step 2, we see that the new extension can be done by
increasing $t$ by at least some $\tilde\delta>0$, with $\tilde\delta$ depends only on $M_0$ and $d$. Furthermore, from the above proof of
\eqref{v-upper-bd} we easily see that the extended solution $v$ satisfies \eqref{v-upper-bd} in the newly extended range of $t$.
Thus the extension by $\tilde \delta$ for $t$ can be repeated. Clearly by repeating this process finitely many times, the solution of
\eqref{201} will be uniquely extended to $t\in [t_x, T)$. As explained above, now
 \eqref{v-upper-bd} holds for $t\in [t_x, T)$, and hence to prove \eqref{v-bound}, it only remains to show $V_{g,h}(t,x)>0$ for $t\in (0, T)$ and $x\in (g(t), h(t))$. However, due to {\rm\bf (f1)-(f2)} and \eqref{v-upper-bd}, we may write $f(t,x,V_{g,h}(t,x))=c(t,x)V_{g,h}(t,x)$ with $c\in L^\infty( \Omega_s)$ for any $s\in (0, T)$. Thus we can use Lemma 2.2 to conclude.
\end{proof}

\vspace{10pt}
\begin{proof}[Proof of Theorem \ref{Thm22}]
By Lemma \ref{Lemma202},
for any $T>0$ and $(h,g)\in\mathbb{G}_{h_0,T}\times\mathbb H_{h_0,T}$, we can find a
unique $V_{g,h}\in\mathbb{X}_{u_0,g,h}$ that solves (\ref{201}), and it has the property
$$
0<V_{g,h}(t,x)\le M_0:=\max\big\{\|u_0\|_\infty,~K_0\big\} \mbox{ for  } (t,x)\in
\Omega_{g,h}.
$$

Using such a $V_{g,h}(t,x)$, we define the mapping $\tilde \Gamma$ by $\tilde \Gamma (g,h)=
\left(\tilde g,\,\tilde h\right)$, where
\begin{equation}
\left\{
\begin{aligned}
&\tilde h(t)=h_0+\mu\int_0^t\int_{g(\tau)}^{h(\tau)}\int_{h(\tau)
}^{+\infty}J(x-y)V_{g,h}(\tau,x)dydxd\tau,\\
&\tilde g(t)=-h_0-\mu\int_0^t\int_{g(\tau)}^{h(\tau)}\int_{-\infty
}^{g(\tau)}J(x-y)V_{g,h}(\tau,x)dydxd\tau
\end{aligned}
\right.
\label{2003}
\end{equation}
for $0<t\leq T$. To simplify notations, we will write
\[
\mathbb G_T=\mathbb G_{h_0,T},\; \mathbb H_T=\mathbb H_{h_0,T},\; \Omega_T=\Omega_{g,h},\; \mathbb X_T=\mathbb X_{u_0, g, h}.
\]

To prove this theorem, we will show
that if $T$ is small enough, then $\tilde\Gamma$ maps  a suitable closed subset $\Sigma_T$ of $\mathbb{G}_T\times\mathbb{H}_T$ into itself,
and  is a contraction mapping. This clearly implies that $\tilde \Gamma$ has a unique fixed point in $\Sigma_T$,
which gives a solution $(V_{g,h}, g, h)$ of \eqref{101} defined for $t\in (0, T]$. We will show that any solution $(u, g,h)$ of \eqref{101}  with  $(g,h)\in \mathbb{G}_T\times\mathbb{H}_T$ must satisfy $(g,h)\in \Sigma_T$, and hence $(g,h)$ must coincide with the unique fixed point of $\tilde\Gamma$ in $\Sigma_T$, which then implies that the solution $(u,g,h)$ of \eqref{101} is unique.  We will finally  show that this unique solution
defined locally in time can be extended uniquely for all $t>0$.

We now carry out this plan in several steps.
\medskip

\noindent
{\bf Step 1:} {\it Properties of $(\tilde g, \tilde h)$ and a closed subset of $\mathbb{G}_T\times\mathbb{H}_T$.}

Let $(g,h)\in \mathbb{G}_T\times\mathbb{H}_T$.
The definitions of $\tilde h(t)$ and
$\tilde g(t)$ indicate that they belong to $C^1([0, T])$  and for $0<t\le T$,
\begin{equation}
\left\{
\begin{aligned}
&\tilde h'(t)= \mu \int_{g(t)}^{h(t)}\int_{h(t)
}^{+\infty}J(x-y)dyV_{g,h}(t,x)dx,\\
&\tilde g'(t)=-\mu\int_{g(t)}^{h(t)}\int_{-\infty
}^{g(t)}J(x-y)dyV_{g,h}(t,x)dx.
\end{aligned}
\right.
\label{2006}
\end{equation}
These identities already imply $\tilde\Gamma (g,h)=(\tilde g, \tilde h)\in  \mathbb{G}_T\times\mathbb{H}_T$, but in order to show $\tilde \Gamma$ is a contraction mapping, we need to
prove some further properties of $\tilde g$ and $\tilde h$, and then choose a suitable closed subset of  $\mathbb{G}_T\times\mathbb{H}_T$, which
is invariant under $\tilde\Gamma$, and on which $\tilde \Gamma$ is a contraction mapping.

Since  $v=V_{g,h}$ solves (\ref{201})  we  obtain by using {\bf (f1)-(f2)} and  \eqref{v-bound} that
\begin{equation}
\left\{
\begin{aligned}
&\left(V_{g,h}\right)_t(t,x)\ge-dV_{g,h}(t,x)-K(M_0)V_{g,h}(t,x), & &0<t\le T,~x\in(g(t),h(t)),\\
&V_{g,h}(t,h(t))=V_{g,h}(t,g(t))=0,& &0<t\le T,\\
&V_{g,h}(0,x)=u_0(x),& &x\in[-h_0,h_0].
\end{aligned}
\right.
\label{2007}
\end{equation}
It follows that
\begin{equation}
\label{V>cu_0}
V_{g,h}(t,x)\ge e^{-(d+K(M_0))t}u_0(x)\ge e^{-(d+K(M_0))T}u_0(x) \mbox{ for } x\in [-h_0, h_0],\; t\in (0, T].
\end{equation}
By \textbf{(J)}  there
exist constants $\epsilon_0\in (0, h_0/4)$ and $\delta_0>0$ such that
\begin{equation}
\label{J>delta_0}
\mbox{$J(x-y)
\ge\delta_0$ if $|x-y|\le\epsilon_0$.}
\end{equation}
Using \eqref{2006} we easily see
\[
[\tilde h(t)-\tilde g(t)]'\leq \mu M_0 [h(t)-g(t)] \mbox{ for } t\in [0, T].
\]
We now assume  that $(g, h)$ has the extra property that
\[\mbox{
    $h(T)-g(T)\leq 2h_0+\frac{\epsilon_0}{4}$.}
    \]
    Then
\[
\tilde h(t)-\tilde g(t)\leq 2h_0 +T\mu M_0(2h_0+\frac{\epsilon_0}{4})\leq 2h_0+\frac{\epsilon_0}{4} \mbox{ for } t\in [0, T],
\]
provided that $T>0$ is small enough, depending on $ (\mu, M_0, h_0, \epsilon_0)$.
We fix such a  $T$ and notice from the above extra assumption on $(g, h)$  that
\[
h(t)\in [h_0,  h_0+ \frac{\epsilon_0}{4}],\; g(t)\in [-h_0-\frac{\epsilon_0}{4}, -h_0]  \mbox{ for } t\in [0, T].
\]
Combining this with \eqref{V>cu_0} and \eqref{J>delta_0} we obtain, for such $T$ and $t\in (0, T]$,
\begin{align*}
\int_{g(t)}^{h(t)}\int_{h(t)}^{+\infty}J(x-y)V_{g,h}
(t,x)dydx&\ge\int_{h(t)-\frac{\epsilon_0}{2}}^{h(t)}\int_{
h(t)}^{h(t)+\frac{\epsilon_0}{2}}J(x-y)V_{g,h}(t,x)dydx\\
&\ge e^{-(d+K(M_0))T}\int_{h_0-\frac{\epsilon_0}{4}}^{h_0}\int_{
h_0+\frac{\epsilon_0}{4}}^{h_0+\frac{\epsilon_0}{2}}J(x-y)dyu_0(x)dx\\
&\ge\frac 14\epsilon_0\delta_0e^{-(d+K(M_0))T}\int_{h_0-\frac{
\epsilon_0}{4}}^{h_0}u_0(x)dx=:c_0>0,
\end{align*}
with $c_0$ depending only on $(J, u_0, f)$. Thus, for sufficiently small $T=T(\mu, M_0, h_0, \epsilon_0)>0$,
\begin{equation}
\label{tilde-h'}
\tilde h'(t)\geq \mu c_0 \mbox{ for } t\in [0, T].
\end{equation}
We can similarly obtain, for such $T$,
\begin{equation}
\label{tilde-g'}
\tilde g'(t)\leq -\mu \tilde c_0 \mbox{ for } t\in [0, T],
\end{equation}
for some positive constant $\tilde c_0$ depending only on $(J, u_0, f)$.

We now define, for $s\in (0, T_0]:=(0, T(\mu, M_0, h_0, \epsilon_0)]$,
\begin{align*}
\Sigma_s:=&\Big\{(g,h)\in\mathbb G_s\times \mathbb H_s: \sup_{0\leq t_1<t_2\leq s}\frac{g(t_2)-g(t_1)}{t_2-t_1}\leq -\mu \tilde c_0,\;
\inf_{0\leq t_1<t_2\leq s}\frac{h(t_2)-h(t_1)}{t_2-t_1}\geq \mu c_0,\\
& \hspace{8cm} h(t)-g(t)\leq 2h_0+\frac{\epsilon_0}{4} \mbox{ for } t\in [0, s]\Big\}.
\end{align*}
Our analysis above shows that
\[
\tilde \Gamma (\Sigma_s)\subset \Sigma_s \mbox{ for } s\in (0, T_0].
\]

\medskip

\noindent {\bf Step 2:} {\it $\tilde\Gamma$ is a
contraction mapping on $\Sigma_s$ for sufficiently small $s>0$.}

Let $s\in (0, T_0]$, $(h_1,g_1), (h_2,g_2)\in\Sigma_s$,
 and note that $\Sigma_s$ is a complete metric space
under the metric
$$
d\left((h_1,g_1),(h_2,g_2)\right)=\|h_1-h_2\|_{C([0,s])}+
\|g_1-g_2\|_{C([0,s])}.
$$
For $i=1,2$, let us denote
\[
\mbox{
 $V_i(t,x)=V_{h_i,g_i}(t,x)$ and
$\tilde\Gamma\left(h_i,g_i\right)=\left(\tilde h_i,\tilde g_i
\right)$. }
\]
We also define
\begin{align*}
&H_1(t)=\min\left\{h_1(t),~h_2(t)\right\},~~~H_2(t)=\max\left\{
h_1(t),~h_2(t)\right\},\\
&G_1(t)=\min\left\{g_1(t),~g_2(t)\right\},~~~~G_2(t)=\max\left\{
g_1(t),~g_2(t)\right\},\\
&\Omega_{*s}=\Omega_{G_1, H_2}=\Omega_{g_1, h_1}\cup \Omega_{g_2, h_2}.
\end{align*}
For $t\in [0,s]$, we have
\[
2h_0\leq H_2(t)-G_1(t)\leq 2h_0+\epsilon_0\leq 3h_0,
\]
and
\begin{align*}
~&\left|\tilde h_1(t)-\tilde h_2(t)\right|
\\
\le~&\mu\int_0^t\left|\int_{g_1(\tau)}^{h_1(\tau)}\int_{h_1(\tau)
}^{+\infty}J(x-y)V_1(\tau,x)dydxd\tau-\int_{g_2(\tau)}^{h_2(\tau)}
\int_{h_2(\tau)}^{+\infty}J(x-y)V_2(\tau,x)dydx\right|d\tau
\\
\le~&\mu\int_0^t\int_{g_1(\tau)}^{h_1(\tau)}\int_{h_1(\tau)
}^{+\infty}J(x-y)\Big|V_1(\tau,x)-V_2(\tau,x)\Big|dydxd\tau
\\
&~+\mu\int_0^t\left|\left(\int_{h_1(\tau)}^{h_2(\tau)}\int_{h_1(\tau)
}^{+\infty}+\int_{g_2(\tau)}^{g_1(\tau)}\int_{h_1(\tau)
}^{+\infty}+\int_{g_2(\tau)}^{h_2(\tau)}\int_{h_1(\tau)}^{h_2(\tau)}\right)
J(x-y)V_2(t,x)dydx\right|d\tau
\\
\le~&3h_0\mu\|V_1-V_2\|_{C(\overline\Omega_{*s})}s+\mu M_0\big(1+3h_0\|J\|_\infty\big)\|h_1-h_2\|_{C([0,s])}s
+\mu M_0\|g_1-g_2\|_{C([0,s])}s\\
\le ~ &C_0s\Big[\|V_1-V_2\|_{C(\overline\Omega_{*s})}+\|h_1-h_2\|_{C([0,s])}+\|g_1-g_2\|_{C([0,s])}\Big],
\end{align*}
where $C_0$ depends only on $(\mu, u_0, J, f)$. Let us recall that $V_i$ is always extended by 0 in $\big([0,\infty)\times \mathbb R\big)\setminus \Omega_{g_i, h_i}$ for $i=1,2$.

Similarly, we have, for $t\in [0, s]$,
\begin{align*}
\Big|\tilde g_1(t)-\tilde g_2(t)\Big|\le C_0s\Big[\|V_1-V_2\|_{C(\overline\Omega_s)}+\|h_1-h_2\|_{C([0,s])}+\|g_1-g_2\|_{C([0,s])}\Big].
\end{align*}
Therefore,
\begin{equation}
\begin{aligned}
&~\|\tilde h_1-\tilde h_2\|_{C([0,s])}+\|\tilde g_1-\tilde g_2
\|_{C([0,s])}\\
\le&~2C_0s\Big[\|V_1-V_2\|_{C(\overline\Omega_{*s})}+\|h_1-h_2\|_{C([0,s])}+\|g_1-g_2\|_{C([0,s])}\Big].
\end{aligned}
\label{20013}
\end{equation}

\medskip

Next, we  estimate $\|V_1-V_2\|_{C(\overline\Omega_{*s})}$. We denote $U=V_1-V_2$, and
for fixed $(t^*,x^*)\in\Omega_{*s}$,  we consider
three cases separately.

\smallskip
\noindent
\underline{Case 1.}  $x^*\in[-h_0, h_0]$.

It
follows from the equations satisfied by $V_1$ and $V_2$ that
$U(0,x^*)=0$ and for $0<t\leq s$,
\begin{equation}
U_t(t,x^*)+c_1(t,x^*)U(t,x^*)=A(t,x^*),
\label{2008}
\end{equation}
where
\begin{align*}
&c_1(t,x^*):=d-\frac{f(t,x^*,V_1(t,x^*))-f(t,x^*,V_2(t,x^*))}
{V_1(t,x^*)-V_2(t,x^*)} \mbox{ and so } \|c_1\|_\infty\leq d+K(M_0),\\
&A(t,x^*):=d\int_{g_1(t)}^{h_1(t)}J(x^*-y)V_1(t,y)dy
-d\int_{g_2(t)}^{h_2(t)}J(x^*-y)V_2(t,y)dy.
\end{align*}
Thus
\begin{align*}
U(t^*,x^*)
=e^{-\int_0^{t^*}c_1(\tau,x^*)d\tau}
\int_0^{t^*}e^{\int_0^tc_1(\tau,x^*)d\tau}A(t,x^*)dt.
\end{align*}
We have
\begin{align*}
\Big|A(t,x^*)\Big|&=d\left|\int_{g_1(t)}^{h_1(t)}J(x^*-y)
V_1(t,y)dy-\int_{g_2(t)}^{h_2(t)}J(x^*-y)V_2(t,y)dy\right|\\
&\le d\int_{g_1(t)}^{h_1(t)}J(x^*-y)\big|V_1(t,y)-V_2(t,y)\big|dy
+d\left|\left(\int_{g_2(t)}^{g_1(t)}+\int_{h_1(t)}^{h_2(t)}\right)J(x^*-y)V_2(t,y)dy\right|\\
&\le d\|U\|_{C(\overline \Omega_{*s})}+d\|J\|_\infty M_0\left[\|h_1-h_2\|_{C([0,s])}+\|g_1-g_2\|_{C([0,s])}\right].
\end{align*}
Thus for some $C_1>0$ depending only on
$(d,u_0, M_0, J)$, we have
\begin{equation}
\max_{t\in [0,s]}\Big|A(t,x^*)\Big|
\le C_1 \left(\|U\|_{C(\overline \Omega_{*s})}+\|h_1
-h_2\|_{C([0,s])}+\|g_1-g_2\|_{C([0,s])}\right).
\label{A}
\end{equation}
It follows that
\begin{equation}
\Big|U(t^*,x^*)\Big|
\le C_1s\, e^{2(d+K(M_0))s} \left(\|U\|_{C(\overline \Omega_{*s})}+\|h_1
-h_2\|_{C([0,s])}+\|g_1-g_2\|_{C([0,s])}\right).
\label{20010}
\end{equation}
\smallskip

 \noindent
  \underline{Case 2.}   $x^*\in(h_0, H_1(s))$.

  In this case there exist $t_1^*,\, t_2^*\in (0, t^*)$ such
that $x^*=h_1(t_1^*)=h_2(t_2^*)$. Without loss of generality, we may assume that $0<t_1^*
\leq t_2^*$.  Now  we use (\ref{2008}) for $t\in [t_2^*, t^*]$, and obtain
\begin{align*}
U(t^*,x^*)=e^{-\int_{t_2^*}^{t^*}c_1(\tau,x^*)d\tau}\left[U(
t_2^*,x^*)+\int_{t_2^*}^{t^*}e^{\int_{t_2^*}^tc_1(\tau,x^*)d\tau}
A(t,x^*)dt\right].
\end{align*}
It follows that
\begin{equation}
\label{U*}
\begin{aligned}
\Big|U(t^*,x^*)\Big|&\le e^{(d+K(M_0))t^*}\left[\Big|U(t_2^*,x^*)
\Big|+\int_{t_2^*}^{t^*}e^{(d+K(M_0))t}\Big|A(t,x^*)\Big|dt\right]\\
&\le e^{(d+K(M_0))s}\Big|U(t_2^*,x^*)\Big|+s e^{2(d+K(M_0))s}\max_{t\in [0, s]}|A(t,x^*)|.
\end{aligned}
\end{equation}
Since $V_1(t_1^*,x^*)=V_2(t_2^*, x^*)=0$, we have
\begin{align*}
U(t_2^*,x^*)=V_1(t_2^*,x^*)-V_1(t_1^*,x^*)=\int_{t_1^*}^{t_2^*}(V_1)_t(t,x^*)dt,
\end{align*}
and hence from the equation satisfied by $V_1$ we obtain
\begin{align*}
\Big|U(t_2^*,x^*)\Big|&\le\int_{t_1^*}^{t_2^*}\left|d\int_{g_1
(t)}^{h_1(t)}J(x^*-y)V_1(t,y)dy
-dV_1(t,x^*)+f(t,x^*,V_1(t,x^*))\right|dt\\
&\le C_2\Big(t_2^*-t_1^*\Big), \mbox{\;\; for some $C_2>0$ depending only on $(d, M_0, f)$}.
\end{align*}
If $t_1^*=t_2^*$ then clearly $U(t_2^*, x^*)=0$. If $t_1^*<t_2^*$, then
using $\frac{h_1(t_2^*)-h_1(t_1^*)}
{t_2^*-t_1^*}\ge \mu c_0$ we obtain
\[
t_2^*-t_1^*\le\Big|h_1(t_2^*)-h_1(t_1^*)\Big|(\mu c_0)^{-1}.
\]
Since
$$
0=h_1(t_1^*)-h_2(t_2^*)=h_1(t_1^*)-h_1(t_2^*)+h_1(t_2^*)-h_2(t_2^*),
$$
we have $h_1(t_2^*)-h_1(t_1^*)=h_1(t_2^*)-h_2(t_2^*)$, and thus
$$
t_2^*-t_1^*\le\Big|h_1(t_2^*)-h_1(t_1^*)\Big|(\mu c_0)^{-1}=
\Big|h_1(t_2^*)-h_2(t_2^*)\Big|(\mu c_0)^{-1}.
$$
Therefore there exists some positive constant $C_3=C_3(\mu c_0,C_2)$
such that
$$
\Big|U(t_2^*,x^*)\Big|\le C_3\|h_1-h_2\|_{C([0,s])}.
$$
Substituting this and \eqref{A} proved in Case 1 above to \eqref{U*}, we obtain
\begin{equation}
\begin{aligned}
\Big|U(t^*,x^*)\Big|
&\le e^{(d+K(M_0))s}C_3\|h_1-h_2\|_{C( [0,s])} \\
& \hspace{.6cm} + C_1 s e^{2(d+K(M_0))s} \left(\|U\|_{C(\overline \Omega_{*s})}+\|h_1
-h_2\|_{C([0,s])}+\|g_1-g_2\|_{C([0,s])}\right).
\end{aligned}
\label{2009}
\end{equation}

\smallskip

\noindent
\underline{Case 3.}  $x^*\in [H_1(s), H_2(s))$.

Without loss of generality we assume that $h_1(s)<h_2(s)$. Then $H_1(s)=h_1(s),\; H_2(s)=h_2(s)$ and
\[
h_1(t^*)\leq h_1(s)<x^*<H_2(t^*) =h_2(t^*),
\]
\[
\mbox{$V_1(t, x^*)=0$ for $t\in [t_2^*, t^*]$,\;  $0<h_2(t^*)-h_2(t_2^*)\leq h_2(t^*)-h_1(t^*)$}.
\]
We have
\begin{align*}
0<V_2(t^*, x^*)&=\int_{t_2^*}^{t^*} \left[d\int_{g_2(t)}^{h_2(t)}J(x^*-y)V_2(t, y)dy-dV_2(t, x^*)+f(t, x^*, V_2(t, x^*))\right]dt\\
&\leq (t^*-t_2^*)\big[d+K(M_0)\big]M_0\\
&\leq \big[h_2(t^*)-h_2(t_2^*)\big](\mu c_0)^{-1}\big[d+K(M_0)\big]M_0\\
&\leq (\mu c_0)^{-1}\big[d+K(M_0)\big]M_0 \big[h_2(t^*)-h_1(t^*)\big]\\
&\leq C_4\|h_1-h_2\|_{C([0,s])},
\end{align*}
with $C_4:= (\mu c_0)^{-1}\big[d+K(M_0)\big]M_0$.

 We thus obtain
\begin{equation}
|U(t^*, x^*)|=V_2(t^*,x^*)\leq C_4\|h_1-h_2\|_{C([0,s])}.
\label{20011}
\end{equation}

The inequalities (\ref{20010}),
(\ref{2009}) and (\ref{20011}) indicate that, there exists $C_5>0$ depending only on $(\mu c_0, d, u_0, J, f)$ such that,
whether we are in Cases 1, 2 or 3, we always have
\begin{equation}
|U(t^*, x^*)|\leq C_5 \left(\|U\|_{C(\overline \Omega_{*s})}s+\|h_1
-h_2\|_{C([0,s])}+\|g_1-g_2\|_{C([0,s])}\right).
\label{20012}
\end{equation}

Analogously, we can examine the cases $x^*\in (G_2(s), -h_0)$ and $x^*\in (G_1(s), G_2(s)]$ to obtain a constant $C_6>0$ depending only on
$(\mu \tilde c_0, d, u_0, J, f)$ such that \eqref{20012} holds with $C_5$ replaced by $C_6$. Setting $C^*:=\max\big\{C_5, C_6\big\}$, we thus obtain
\[
|U(t^*, x^*)|\leq C^* \left(\|U\|_{C(\overline \Omega_{*s})}s+\|h_1
-h_2\|_{C([0,s])}+\|g_1-g_2\|_{C([0,s])}\right) \mbox{ for all } (t^*, x^*)\in\Omega_{*s}.
\]
It follows that
\[
\|U\|_{C(\overline \Omega_{*s})}\leq  C^* \left(\|U\|_{C(\overline \Omega_{*s})}s+\|h_1
-h_2\|_{C([0,s])}+\|g_1-g_2\|_{C([0,s])}\right).
\]
Let us recall that the above inequality holds for all $s\in (0, T_0]$ with $T_0$ given near the end of Step 1. Set $T_1:=\min\Big\{T_0,\; \frac{1}{2C^*}\Big\}$.
Then we easily deduce
\[
\|U\|_{C(\overline \Omega_{*s})}\leq  2C^* \left(\|h_1
-h_2\|_{C([0,s])}+\|g_1-g_2\|_{C([0,s])}\right)\; \mbox{ for } s\in (0, T_1].
\]
Substituting this inequality into (\ref{20013}) we obtain, for $s\in (0, T_1]$,
\begin{align*}
&~\|\tilde h_1-\tilde h_2\|_{C([0,s])}+\|\tilde g_1-\tilde g_2
\|_{C([0,s])}\\
\le&~ 2C_0(2C^*+1)s\left[\|h_1-h_2\|_{C([0,s])}
+\|g_1-g_2\|_{C([0,s])}\right].
\end{align*}
Thus if we define
$T_2$  by $2C_0(2C^*+1)T_2=\frac 12$, and $T^*:=\min\big\{T_1, T_2\big\}$, then
\begin{align*}
\|\tilde h_1-\tilde h_2\|_{C([0,T^*])}+\|\tilde g_1-\tilde g_2
\|_{C([0,T^*])}\le\frac 12\left[\|h_1-h_2\|_{C([0,T^*])}
+\|g_1-g_2\|_{C([0,T^*])}\right],
\end{align*}
 i.e., $\tilde\Gamma$
is a contraction mapping on $\Sigma_{T^*}$.

\medskip

\noindent
{\bf Step 3:} {\it Local existence and uniqueness.}

By Step 2 and the Contraction Mapping Theorem we know that \eqref{101} has a solution $(u, g,h)$ for $t\in (0, T^*]$. If we can show that
$(g,h)\in \Sigma_{T^*}$  holds for any solution $(u,g,h)$ of \eqref{101} defined over $t\in (0, T^*]$, then it is the unique fixed point of $\tilde \Gamma$ in $\Sigma_{T^*}$ and the uniqueness of $(u,g,h)$ follows.

So let $(u,g,h)$ be an arbitrary solution of \eqref{101} defined for $t\in (0, T^*]$. Then
\[
\left\{
\begin{aligned}
& h'(t)= \mu \int_{g(t)}^{h(t)}\int_{h(t)
}^{+\infty}J(x-y)dyu (t,x)dx,\\
& g'(t)=-\mu\int_{g(t)}^{h(t)}\int_{-\infty
}^{g(t)}J(x-y)dyu(t,x)dx.
\end{aligned}
\right.
\]
By Lemma \ref{Lemma202}, we have
\[
0<u(t,x)\leq M_0 \mbox{ for } t\in [0, T^*], \; x\in (g(t), h(t)).
\]
It follows that
\[
[h(t)-g(t)]'=\mu\int_{g(t)}^{h(t)}\left[1-\int_{g(t)}^{h(t)}J(x-y)dy\right]u(t,x)dx\leq \mu M_0[h(t)-g(t)] \mbox{ for } t\in (0, T^*].
\]
We thus obtain
\begin{equation}
\label{h-g}
h(t)-g(t)\leq 2h_0 e^{\mu M_0 t} \mbox{ for } t\in (0, T^*].
\end{equation}
Therefore if we shrink $T^*$  if necessary so that
\[
2h_0e^{\mu M_0 T^*}\leq 2h_0+\frac{\epsilon_0}{4},
\]
then
\[
h(t)-g(t)\leq 2h_0+\frac{\epsilon_0}{4} \mbox{ for } t\in [0, T^*].
\]
Moreover, the proof of \eqref{tilde-h'} and \eqref{tilde-g'} gives
\[
h'(t)\geq \mu c_0,\; g'(t)\leq -\mu \tilde c_0 \mbox{ for } t\in (0, T^*].
\]
Thus indeed $(g,h)\in\Sigma_{T^*}$, as we wanted.
This proves the local existence and uniqueness of the  solution to \eqref{101}.

\medskip

\noindent
{\bf Step 4:} {\it Global existence and uniqueness.}

By Step 3, we see the \eqref{101} has a unique solution $(u,g,h)$  for some initial time interval $(0, T)$, and for
any $s\in (0, T)$, $u(s,x)>0$ for $x\in (g(s), h(s))$ and $u(s,\cdot)$ is continuous over $[g(s), h(s)]$.
This implies that we can treat $u(s,\cdot)$ as an initial function and use Step 3 to extend the solution from $t=s$ to some $T'\geq T$.
Suppose $(0, \hat T)$ is the maximal interval that the solution $(u,g,h)$ of \eqref{101} can be defined through this extension process.
We show that $\hat T=\infty$. Otherwise $\hat T\in (0, \infty)$ and we are going to derive a contradiction.

Firstly we notice that \eqref{h-g} now holds for $t\in (0, \hat T)$. Since $h(t)$ and $g(t)$ are monotone functions over $[0, \hat T)$,
we may define
\[
h(\hat T):=\lim_{t\to\hat T} h(t),\; g(\hat T):=\lim_{t\to\hat T} g(t) \; \mbox{ with } h({\hat T})-g(\hat T)\leq 2h_0e^{\mu M_0 \hat T}.
\]
The third and fourth equations in \eqref{101}, together with $0\leq u\leq M_0$ indicate that $h'$ and $g'$ belong to $L^\infty([0, \hat T))$ and
hence with $g(\hat T)$ and $h(\hat T)$ defined as above, $g, h\in C([0,\hat T])$. It also follows that the right-hand side of the first equation in \eqref{101}
belongs to $L^\infty(\Omega_{\hat T})$, where $\Omega_{\hat T}:=\big\{(t,x): t\in [0, \hat T],\; g(t)< x<h(t)\big\}$.
It follows that $u_t\in L^\infty(\Omega_{\hat T})$. Thus for each $x\in (g(\hat T), h(\hat T))$,
\[
u(\hat T,x):=\lim_{t\nearrow \hat T}u(t,x) \mbox{ exists},
\]
and $u(\cdot, x)$ is continuous at $t=\hat T$. We may now view $u(t,x)$ as the unique solution of the ODE problem in Step 1 of the proof of Lemma 2.3
(with $\phi=u$),
which is defined over $[t_x, \hat T]$. Since $t_x$, $J(x-y)$ and $f(t,x,u)$ are all continuous in $x$, by the continuous dependence of the ODE solution to
the initial function and the parameters in the equation, we see that $u(t,x)$ is continuous in $\Omega_{\hat T}$.
By assumption, $u\in C(\overline \Omega_s)$ for any $s\in (0, \hat T)$. To show this also holds with $s=\hat T$, it remains to show that
$u(t,x)\to 0$ as $(t,x)\to (\hat T, g(\hat T))$ and as $(t,x)\to (\hat T, h(\hat T))$ from $\Omega_{\hat T}$. We only prove the former as the other case can be shown similarly.
We note that as $x\searrow g(\hat T)$, we have $t_x\nearrow \hat T$, and so
\begin{align*}
|u(t,x)|&=\left|\int_{t_x}^t \left[ d\int_{g(t)}^{h(t)}J(x-y)u(\tau, y)dy-d u(\tau, x)+f(\tau, x, u(\tau,x))\right]d\tau\right|\\
&\leq (t-t_x)\big[2d+K(M_0)\big]M_0\\
&\to 0 \mbox{ as } \Omega_{\hat T} \ni (t,x)\to (\hat T, g(\hat T)).
\end{align*}

Thus we have shown that $u\in C(\overline \Omega_{\hat T})$ and $(u,g,h)$ satisfies \eqref{101} for $t\in (0, \hat T]$. By Lemma 2.2 we have
$u(\hat T, x)>0$ for $x\in (g(\hat T), h(\hat T))$. Thus we can regard $u(\hat T, \cdot)$ as an initial function and apply Step 3 to conclude that
the solution of \eqref{101} can be extended to some $(0, \tilde T)$ with $\tilde T>\hat T$. This contradicts the definition of $\hat T$. Therefore we must have $\hat T=\infty$.
\end{proof}

\section{Long-time behavior of \eqref{101}: Spreading-vanishing dichotomy}
\subsection{Some preparatory results}
\subsubsection{\underline{Comparison Principles}}
\begin{theorem}\label{thm-CP}
$($Comparison principle$)$ Assume that \textbf{(J)} and  \textbf{(f1)-(f2)} hold, and $u_0$ satisfies \eqref{102}.
For $T\in(0,+\infty)$, suppose that $\overline h,\overline
g\in C([0,T])$ and
$\overline u\in  C\left(\overline \Omega_{\overline g, \overline h, T}\right)$ satisfies
\begin{equation}\label{CP-upper}
\left\{
\begin{aligned}
&\overline u_t\ge d\int_{\overline g(t)}^{\overline h(t)}J(x-y)
\overline u(t,y)dy-d\overline u+f(t,x,\overline u) & &0<t\le T,~x\in(\overline g(t),\overline h(t)),\\
&\overline u(t,\overline g(t))\geq 0, \ \overline u(t,\overline h(t))\geq 0 & &0<t\le T,\\
&\overline h'(t)\ge\mu\int_{\overline g(t)}^{\overline h(t)}\int_{\overline h(t)}^{+\infty}J(x-y)\overline u(t,x)dydx  & &0<t\le T,\\
&\overline g'(t)\le-\mu\int_{\overline g(t)}^{\overline h(t)}\int_{-\infty}^{\overline g(t)}J(x-y)\overline u(t,x)dydx  & &0<t\le T,\\
&\overline u(0,x)\ge u_0(x),~\overline h(0)\ge h_0,~\overline g(0)\le-h_0  & & x\in[-h_0,h_0].
\end{aligned}
\right.
\end{equation}
Then the unique positive solution $(u,g,h)$ of problem
\eqref{101} satisfies
\begin{equation}
u(t,x)\leq\overline u(t,x),~g(t)\geq\overline g(t)~\text{ and
}~h(t)\leq\overline h(t)~\text{ for }~0<t\leq T~\text{ and }~x
\in\mathbb{R}.
\label{CP-compare}
\end{equation}
\end{theorem}

The triplet $(\overline u,\overline g,\overline h)$  above
is called an upper solution of \eqref{101}. We can define a lower
solution and obtain analogous results by reversing all the inequalities
in (\ref{CP-upper}) and (\ref{CP-compare}).

\begin{proof}
First of all,   thanks to (\ref{102}) and Lemma \ref{lemma-MP}, one sees that $\overline u>0$ for $0<t\leq T,~\overline g(t)<x<\overline h(t)$,  and thus  both $\overline h$ and $-\overline g$ are strictly increasing.

For small $\epsilon>0$, let $(u_\epsilon,g_\epsilon,h_\epsilon)$
denote the unique solution of (\ref{101}) with $h_0$ replaced by $h_0^\epsilon:=h_0(1-\epsilon)$, $\mu$ replaced by $\mu_\epsilon
=\mu(1-\epsilon)$, and $u_0$ replaced by $u_0^\epsilon\in C([-
h_0^\epsilon,h_0^\epsilon])$ which satisfies $0\le u_0^\epsilon
(x)< u_0(x)$ in $[-h_0^\epsilon,h_0^\epsilon]$ and $u_0^\epsilon
\left(\frac{h_0}{h_0^\epsilon}x\right)\rightarrow u_0(x)$ as $\epsilon\rightarrow0$ in the  $C([-h_0,h_0])$ norm.

We claim that {\it $h_\epsilon(t)<\overline h(t)$ and $g_\epsilon(t)
>\overline g(t)$ for all $t\in(0,T]$.} Clearly, these hold true for
small $t>0$. Suppose that there exists
$t_1\le T$ such that
$$
h_\epsilon(t)<\bar h(t),\ g_\epsilon(t) >\overline g(t)\  \textrm{for}\  t\in(0,t_1)\  \textrm{and} \ [h_\epsilon(t_1) -\overline h(t_1)][  g_\epsilon(t_1)- \overline g(t_1)]=0.
$$
Without loss of generality, we may assume that
\[
 h_\epsilon(t_1) =\overline h(t_1) \mbox{ and }  g_\epsilon(t_1)\geq  \overline g(t_1).
 \]

We now compare $u_\epsilon$ and $\overline u$ over the region
$$
\Omega_{\epsilon, t_1}:=\left\{(t,x)\in\mathbb{R}^2:0<t\leq t_1,
~g_\epsilon(t)<x<h_\epsilon(t)\right\}.
$$
Let  $w(t,x)=e^{k_1 t}\left(\overline u-u_\epsilon\right)$, where
$k_1>0$ is a constant to be determined later. Then for all $(t,x)
\in\Omega_{\epsilon, t_1}$, there is
\begin{equation}
w_t\ge d\int_{g_\epsilon(t)}^{h_\epsilon(t)}J(x-y)w(t,y)dy
+\left[k_1-d+C(t,x)\right]w(t,x),
\label{304}
\end{equation}
for some  $L^\infty$ function $C(t, x)$.
Choosing $k_1$ large such that $p_1(t,x):=k_1-d+C(t,x)>0$
for all $(t,x)\in\Omega_{\epsilon, t_1}$.
By Lemma \ref{lemma-MP},  it follows that $\overline u-u_\epsilon>0$ in  $\Omega_{\epsilon, t_1}$.

Furthermore, according to the definition of $t_1$, we have   $h_\epsilon'(t_1) \ge\overline h'(t_1)$. Thus
\begin{eqnarray*}
0 &\ge & \overline h'(t_1)-h_\epsilon'(t_1)\\
  & \geq & \mu\int_{\overline g(t_1)
}^{\overline h(t_1)}\int_{\overline h(t_1)}^{+\infty}J(x-y)
 \overline u(t_1,x) dydx - \mu_{\epsilon}\int_{  g_{\epsilon}(t_1)
}^{  h_{\epsilon}(t_1)}\int_{ h_{\epsilon}(t_1)}^{+\infty}J(x-y)
  u_{\epsilon}(t_1,x) dydx\\
  &>& \mu_{\epsilon}\int_{  g_{\epsilon}(t_1)
}^{  h_{\epsilon}(t_1)}\int_{ h_{\epsilon}(t_1)}^{+\infty}J(x-y)
 \big[\overline u(t_1,x)- u_{\epsilon}(t_1,x)\big] dydx>0,
\end{eqnarray*}
which is a contradiction.  The claim is thus proved, i.e., we always have $h_\epsilon(t)<\overline h(t)$ and $g_\epsilon(t)
>\overline g(t)$ for all $t\in(0,T]$. Then the above arguments yield that $\overline u(t,x)>u_\epsilon(t,x)$  in $\Omega_{\epsilon, T}$.

Since the unique solution of
(\ref{101}) depends continuously on the parameters
  in  (\ref{101}), the desired result
then follows by letting $\epsilon\rightarrow0$.
\end{proof}

The following result is a direct consequence of the comparison
principle, where to stress the dependence on the parameter $\mu$, we use $\left(u^\mu,g^\mu,
h^\mu\right)$ to denote the solution of problem (\ref{101}).

\begin{corollary}\label{corollary-mu-increasing} Assume that \textbf{(J)} and  \textbf{(f1)-(f2)} hold, and $u_0$ satisfies \eqref{102}.
   If $\mu_1\le\mu_2$, we have $h^{\mu_1}
(t)\le h^{\mu_2}(t)$, $g^{\mu_1}(t)\ge g^{\mu_2}(t)$ for
$t>0$, and $u^{\mu_1}(t,x)\le u^{\mu_2}(t,x)$ for $t>0$
and $g^{\mu_1}(t)<x<h^{\mu_1}(t)$.
\end{corollary}

{\begin{lemma}\label{lemma-MP2}
Assume that {\rm \bf (J)}   holds, and  $h_0, T>0$.
Suppose that
 $u(t,x)$ as well as $u_t(t,x)$ are
continuous in $\Omega_0:=[0, T]\times [-h_0, h_0]$, and  for some $c\in L^\infty (\Omega_0)$,
\begin{equation}\label{lemma-MP2-u}
\left\{
\begin{aligned}
&u_t(t,x)\ge d\int_{-h_0}^{h_0}J(x-y)u(t,y)dy-du +c(t,x)u, && t\in (0, T],\  x\in [-h_0, h_0],\\
&u(0,x)\ge0,  && x\in [-h_0,  h_0].
\end{aligned}
\right.
\end{equation}
Then $u(t,x)\ge0$ for all $0\le t\le T$ and $x\in[-h_0, h_0]$.
Moreover, if $u(0,x)\not\equiv0$ in $[-h_0, h_0]$, then $u(t,x)>0$ in $(0, T]\times [-h_0, h_0]$.
\end{lemma}}
\begin{proof}
This result is well known, and can be proved by  the arguments in the proof of Lemma \ref{lemma-MP}; the situation here is actually much simpler.
We omit the details.
\end{proof}

\subsubsection{\underline{Some related eigenvalue problems}}
Here we recall and prove some results on the principal eigenvalue of the linear
 operator $\mathcal{L}_{\Omega}+a: C(\overline
\Omega)\mapsto C(\overline\Omega)$ defined by
$$
\left(\mathcal{L}_{\Omega}+a\right)[\,\phi\,] (x):= d \left[\int_{\Omega}J(x-y)
\phi(y)dy -  \phi(x)\right]+a(x)\phi(x),
$$
where $\Omega$ is an open interval in $\mathbb{R}$, possibly unbounded,
$a\in C(\overline \Omega)$ and $J$ satisfies \textbf{(J)}.

Define
$$
\lambda_p  (\mathcal{L}_{\Omega}+a):=  \inf\ \Big\{  \lambda\in\mathbb R  :  (\mathcal{L}_{\Omega}+a)[\,\phi\,]\leq \lambda \phi \mbox{ in } \Omega \mbox{ for some }
\phi\in C(\overline\Omega), \phi>0 \Big\}.
$$
As usual, if $\lambda_p  (\mathcal{L}_{\Omega}+a(x))$ is an eigenvalue of the  operator $\mathcal{L}_{\Omega}+a$ with a continuous and positive eigenfunction, we call it
a {\it principal eigenvalue}.

In this paper, we are particularly interested in the properties of $\lambda_p(\mathcal{L}_{(\ell_1, \ell_2)}+  a_0)$, with $a_0$  a positive constant and $-\infty\leq \ell_1<  \ell_2\leq +\infty$. In this special case, it is well known (see, e.g., \cite{BerestyckiJFA2016, CovilleJDE2010, LiFangDCDS2017}) that $\lambda_p(\mathcal{L}_{(\ell_1, \ell_2)}+  a_0)$ is a principal eigenvalue. Moreover, we show that the following conclusions hold.

\begin{proposition}\label{prop-EV}
Assume that the kernel $J$ satisfies  \textbf{(J)}, $a_0$ is a positive constant and $-\infty<\ell_1 < \ell_2<+\infty$. Then the following hold true:
\begin{itemize}
\item[(i)] $\lambda_p(\mathcal{L}_{(\ell_1, \ell_2)}+  a_0)$ is strictly increasing and continuous in $\ell:=\ell_2-\ell_1$,
\item[(ii)]  $\lim_{\ell_2-\ell_1\to+\infty}\lambda_p(\mathcal{L}_{(\ell_1, \ell_2)}+  a_0)=a_0$,
\item[(iii)] $\lim_{\ell_2-\ell_1\to 0}\lambda_p(\mathcal{L}_{(\ell_1, \ell_2)}+  a_0)=a_0-d$.
\end{itemize}
\end{proposition}

\begin{proof} Since $a_0$ is a constant, it follows easily from the definition that $\lambda_p(\mathcal{L}_{(\ell_1, \ell_2)}+  a_0)$ depends only on $\ell:=\ell_2-\ell_1$.
So we only need to prove the stated conclusions for $(\ell_1, \ell_2)=(0, \ell)$.

(i) Suppose $\hat \ell>\ell$ and denote
\[
\lambda_p:=\lambda_p(\mathcal{L}_{(0, \ell)}+  a_0),\; \hat \lambda_p:=\lambda_p(\mathcal{L}_{(0, \hat\ell)}+  a_0).
\]
To prove the monotonicity in $\ell$ it suffices to show $\lambda_p<\hat\lambda_p$.
Let  $\phi$ be a positive eigenfunction corresponding to  $\lambda_p$.
Then  $\phi$ is positive and continuous over $[0, l]$, and
\[
(\mathcal{L}_{(0, \ell)}+  a_0)[\, \phi\,](x)=\lambda_p \phi(x) \mbox{ for } x\in [0,\ell].
\]
Since
\[
\int_0^\ell J(x-y)\phi(y)dy>0,
\]
it follows from the above identity that
\[
c_0:=\lambda_p+d-a_0>0.
\]
Define
\[
\tilde\phi(x):=\frac d{c_0}\int_0^\ell J(x-y)\phi(y)dy \mbox{ for } x\in [0, \hat \ell].
\]
 Clearly $\tilde\phi(x)=\phi(x)$ for $x\in [0, \ell]$ and $\tilde \phi(x)>0$ when $x-\ell>0$ is small.

By the variational characterization of $\hat\lambda_p(\mathcal{L}_{(0, \hat\ell)}+  a_0)$
(see, e.g., \cite{BerestyckiJFA2016}), we have
\[
\hat\lambda_p=\sup_{0\not\equiv \psi\in L^2([0,\hat\ell])}  \frac{\displaystyle d \int_{0}^{\hat\ell} \int_{0}^{\hat\ell} J(x-y) \psi(y)\psi(x) dydx}{\displaystyle\int_0^{\hat\ell} \psi^2(x)dx}- d  + a_0.
\]
It follows that
\[
\hat\lambda_p+d-a_0\geq \tilde \lambda:= \frac{\displaystyle d \int_{0}^{\hat\ell} \int_{0}^{\hat\ell} J(x-y) \tilde\phi(y)\tilde \phi(x) dydx}{\displaystyle\int_0^{\hat\ell} \tilde\phi^2(x)dx}.
\]
Clearly
\[
\begin{array}{ll}
\displaystyle d \int_{0}^{\hat\ell} \int_{0}^{\hat\ell} J(x-y) \tilde\phi(y)\tilde\phi(x) dydx&=\displaystyle\int_0^{\hat\ell}\Big[c_0\tilde\phi(x)
+d\int_{\ell}^{\hat\ell}J(x-y) \tilde\phi(y)dy\Big]\tilde\phi(x) dx\\
&>c_0\displaystyle\int_0^{\hat\ell}\tilde\phi^2(x)dx.
\end{array}
\]
Therefore
\[
\tilde\lambda>c_0=\lambda_p+d-a_0
\]
and so
$
\hat\lambda_p>\lambda_p$.
This proves the monotonicity.

We next prove the continuity in $\ell$. By Lemma 2.4 in  \cite{BerestyckiJFA2016} we have, for any $\ell_0>0$,
\[
\lim_{\ell\nearrow\ell_0}\lambda_p(\mathcal{L}_{(0, \ell)}+  a_0)=\lambda_p(\mathcal{L}_{(0, \ell_0)}+  a_0).
\]
It remains to show
\begin{equation}\label{r-lim}
\lim_{\ell\searrow\ell_0}\lambda_p(\mathcal{L}_{(0, \ell)}+  a_0)=\lambda_p(\mathcal{L}_{(0, \ell_0)}+  a_0).
\end{equation}
Denote $\lambda_p^0:=\lambda_p(\mathcal{L}_{(0, \ell_0)}+  a_0)$ and let $\phi_0(x)>0$ be a corresponding eigenfunction of $\lambda_p^0$.
We extend $\phi_0(x)$  by defining $\phi_0(x)=\phi_0(\ell_0)$ for $x\geq \ell_0$, and denote the extended $\phi_0$ still by itself. We claim that for any $\epsilon>0$, there exists $\delta>0$ small so that for all $\ell\in (\ell_0, \ell_0+\delta)$,
\[
(\mathcal{L}_{(0, \ell)}+  a_0)[\,\phi_0\,](x)\leq (\lambda_p^0+\epsilon)\phi_0(x) \mbox{ for } x\in [0,\ell].
\]
Clearly this implies $\lambda_p^0<\lambda_p(\mathcal{L}_{(0, \ell)}+  a_0)\leq \lambda_p^0+\epsilon $ for such $\ell$, and \eqref{r-lim} is a consequence of this conclusion.
So to complete the proof of the continuity on $\ell$, it suffices to prove the above claim.

We have, for $\ell>\ell_0$,
\begin{eqnarray*}
(\mathcal{L}_{(0, \ell)}+  a_0)[\,\phi_0\,](x)&=& d\int_0^\ell J(x-y)\phi_0(y)dy+(a_0-d)\phi_0(x)\\
&=&d\int_0^{\ell_0}J(x-y)\phi_0(y)dy+(a_0-d)\phi_0(x)+d\int_{\ell_0}^\ell J(x-y)\phi_0(\ell_0)dy
\end{eqnarray*}
Clearly
\[
0\leq d\int_{\ell_0}^\ell J(x-y)\phi_0(\ell_0)dy\leq d\|J\|_\infty \phi_0(\ell_0)(\ell-\ell_0)<\epsilon \mbox{ for } \ell\in (\ell_0,\ell_0+\delta^1_\epsilon),\; x\in [0,\ell],
\]
where $\delta_\epsilon^1:=\epsilon/[d\|J\|_\infty \phi_0(\ell_0)]$.
Moreover,
\[
\tilde \phi_0(x):=d\int_0^{\ell_0}J(x-y)\phi_0(y)dy+(a_0-d)\phi_0(x)=\lambda_p^0\,\phi_0(x) \mbox{ for } x\in [0,\ell_0],
\]
and $\tilde\phi_0(x)$ is a continuous function for $x\in [0,+\infty)$. Therefore, for sufficiently small $\delta^2_\epsilon>0$,
\[
|\tilde \phi_0(x)-\lambda_p^0\,\phi_0(\ell_0) |<\epsilon \mbox{ for } x\in (\ell_0, \ell_0+\delta^2_\epsilon).
\]
It follows that
\[
(\mathcal{L}_{(0, \ell)}+  a_0)[\,\phi_0\,](x)\leq \lambda_p^0\,\phi_0(x)+2\epsilon \leq \left[\lambda_p^0+\frac{2\epsilon}{\min_{x\in [0,\ell_0]}\phi_0(x)}\right]\phi_0(x)
\mbox{ for } x\in [0,\ell],
\]
provided that $\ell\in (\ell_0, \ell_0+\delta_\epsilon)$, with $\delta_\epsilon:=\min\{\delta_\epsilon^1,\delta_\epsilon^2\}$. This clearly implies our claim and the continuity of
$\lambda_p(\mathcal{L}_{(0, \ell)}+  a_0)$ on $\ell$ is thus proved.

(ii) 
Due to $\int_{\mathbb R} J(x)dx=1$, we have
\[
\int_{0}^{\ell} \int_{0}^{\ell} J(x-y) \phi(y)\phi(x) dydx\leq \int_{0}^{\ell} \int_{0}^{\ell} J(x-y)\frac{\phi^2(y)+\phi^2(x)}{2}dydx\leq \int_0^\ell \phi^2(x) dx.
\]
It then follows from the above variational characterization that
\[
\lambda_p(\mathcal{L}_{(0, \ell)}+  a_0)\leq d-d+a_0=a_0.
\]

By \textbf{(J)}, for any small $\epsilon>0$, there exists $L=L_\epsilon>0$ such that
\[
\int_{-L}^LJ(x)dx>1-\epsilon.
 \]
  Then taking $\phi\equiv 1$ as the test function in the variational characterization of
$\lambda_p(\mathcal{L}_{(0, \ell)}+  a_0)$ we obtain, for all large $\ell>0$,
\begin{eqnarray*}
\lambda_p(\mathcal{L}_{(0, \ell)}+  a_0) &\geq&  \frac{d\displaystyle \int_{0}^{\ell} \int_{0}^{\ell} J(x-y)  dydx}{ \ell }-d+a_0 \\
& \geq &\frac{d\displaystyle\int_{L}^{\ell-L}\int_{0}^{\ell} J(x-y)  dydx}{ \ell } -d+a_0\\
&\geq & \frac{d\left(\ell-2L \right)\displaystyle\int_{-L}^L J(\xi)d\xi}{ \ell }-d+a_0\\
&\geq& \frac{d\left(\ell-2L \right)(1-\epsilon)}{ \ell }-d+a_0\\
& \to& -\epsilon d+a_0 \;\;\;\; \mbox{ as } \ell\to+\infty.
\end{eqnarray*}
Hence
\[
\liminf_{\ell\to+\infty}\lambda_p(\mathcal{L}_{(0, \ell)}+  a_0)\geq -\epsilon d+a_0.
\]
Since $\epsilon>0$ can be arbitrarily small, it follows that
\[
\liminf_{\ell\to+\infty}\lambda_p(\mathcal{L}_{(0, \ell)}+  a_0)\geq a_0,
\]
which together with $\lambda_p(\mathcal{L}_{(0, \ell)}+  a_0)\leq a_0$ proves the desired result.

(iii) We want to show  that
\begin{equation}\label{pf-prop-h-zero}
\lim_{h\rightarrow 0^+} \lambda_p(\mathcal{L}_{(0,  h)}+  a_0) = a_0 -d.
\end{equation}
Since $\lambda_h := \lambda_p(\mathcal{L}_{(0, h)}+  a_0)$ is a principal eigenvalue,   there exists  a strictly positive function $\phi_h \in C([0,h])$ such that
$$
d \int_{0}^h J(x-y) \phi_h(y)dy - d \phi_h(x)+a_0 \phi_h(x)  = \lambda_h \phi_h\ \ \  \textrm{in}\ [0,h].
$$
Therefore
\begin{eqnarray*}
\big| \lambda_h  -a_0 +d \big|& =& \frac{d\displaystyle \int_{0}^h \int_{0}^h J(x-y) \phi_h(y) \phi_h (x)dy dx  }{ \displaystyle\int_{0}^h \phi_h^2 (x)dx} \leq  \frac{d \displaystyle\|J\|_\infty \left( \int_{0}^h  \phi_h (x)dx \right)^2 }{\displaystyle \int_{0}^h \phi_h^2 (x)dx} \\
&\leq & \frac{d \|J\|_\infty h \displaystyle \int_{0}^h  \phi_h^2 (x)dx  }{ \displaystyle\int_{0}^h \phi_h^2 (x)dx}=d\|J\|_\infty h\to 0\;\; \mbox{ as $h \rightarrow 0^+$.}
\end{eqnarray*}
 This proves \eqref{pf-prop-h-zero}.
\end{proof}

\subsubsection{\underline{Some  nonlocal problems over fixed spatial domains}}
We now recall some well known conclusions for nonlocal
diffusion equations over  fixed spatial domains. We first consider the problem
\begin{equation}\label{fixedbdry}
\left\{
\begin{aligned}
&u_t = d\left[\int_{\ell_1}^{\ell_2}J(x-y)u(t,y)dy-u(t,x)\right] + f(u),&  &  t>0,~x\in[\ell_1,\ell_2],\\
& u(0,x)=u_0(x),&  & x\in[\ell_1,\ell_2],
\end{aligned}
\right.
\end{equation}
where $-\infty < \ell_1<\ell_2< +\infty$.

\begin{proposition}[\cite{BatesJMAA2007, CovilleJDE2010}]\label{prop-single}
Denote $\Omega = (\ell_1, \ell_2)$ and suppose  \textbf{(J)} and \textbf{(f1)}--\textbf{(f4)} hold. Then \eqref{fixedbdry} admits a unique positive steady state $u_{\Omega}$ in $C(\bar\Omega)$ if and only if
$$
\lambda_p(\mathcal{L}_{\Omega}+ f'(0))>0.
$$
Moreover, for $u_0(x) \in C(\bar \Omega)$  and $u_0\ge,\not\equiv0$, \eqref{fixedbdry} has a unique solution $u(t,x)$ defined for all $t>0$,
and it converges to $u_{\Omega}$ in $C(\bar\Omega)$ as $t\rightarrow +\infty$ when $\lambda_p(\mathcal{L}_{\Omega}+ f'(0))>0$;
when $\lambda_p(\mathcal{L}_{\Omega}+ f'(0))\leq 0$,  $u(t,x)$ converges to  $0$ in $C(\bar\Omega)$ as $t\rightarrow +\infty$.
\end{proposition}

We note that when {\bf (f1)-(f4)} hold, the function $f(u)$ has a unique positive zero $v_0\in (0, K_0)$.

\begin{proposition}\label{prop-whole}
Assume   \textbf{(J)} and \textbf{(f1)}--\textbf{(f4)}  hold. Then there exists $L>0$ such that for every interval $(\ell_1,\ell_2)$ with length $\ell_2-\ell_1>L$,
we have $\lambda_p(\mathcal{L}_{(\ell_1, \ell_2)}+ f'(0))>0$ and hence \eqref{fixedbdry} has a unique positive steady state $u_{(\ell_1,\ell_2)}$; moreover,
 \begin{equation}\label{lim-u}
 \lim_{-\ell_1,\ell_2\to+\infty} u_{(\ell_1,\ell_2)}=v_0\;\; \mbox{ locally uniformly in $\mathbb{R}$.}
 \end{equation}
\label{proposition37}
\end{proposition}

\begin{proof}
Since
$f'(0)>0$, by Lemma 2.4 of \cite{BerestyckiJFA2016} and Proposition \ref{prop-EV}, we have
$$
 \lim_{\ell_2-\ell_1\rightarrow +\infty} \lambda_p(\mathcal{L}_{(\ell_1,\ell_2)}+ f'(0))=f'(0)>0,
$$
and so there exists $L>0$ large such that
\[
 \lambda_p(\mathcal{L}_{(\ell_1, \ell_2)}+ f'(0))>0\;\; \mbox{ whenever } \ell_2-\ell_1>L.
\]

Fix a positive function $u_0\in L^\infty(\mathbb R)\cap C(\mathbb R)$. By Lemma \ref{lemma-MP2} we can use a simple comparison argument to
show that the unique solution $u_{(\ell_1,\ell_2)}(t,x)$ of \eqref{fixedbdry} has the property
\[
\mbox{$u_{I_1}(t,x)\leq u_{I_2}(t,x)$ for $t>0$ and $x\in I_1$
if }
I_1:=(\ell_1^1,\ell_2^1)\subset I_2:=(\ell_1^2, \ell_2^2).
\]
Letting $t\to+\infty$, it follows that, when $|I_1|>L$,
\[
\mbox{$u_{I_1}(x)\leq u_{I_2}(x)$ for $x\in I_1$
if }
I_1:=(\ell_1^1,\ell_2^1)\subset I_2:=(\ell_1^2, \ell_2^2).
\]
This monotonicity property implies that to show \eqref{lim-u}, it sufficies to prove it along any particular sequence $(\ell_1^n,\ell_2^n)$
with $\ell_1^n\to-\infty,\;\ell_2^n\to+\infty$ as $n\to\infty$.

Fix $z_1, z_2\in \mathbb R$ and consider $(z_i-n,z_i+n)$, $ i=1,2.$
There exits $N \in \mathbb N$ such that for $n\geq N$,
$$
\lambda_p(\mathcal{L}_{(z_i-n,z_i+n)}+ f'(0))>0, \ i=1,2.
$$
 Due to Proposition \ref{prop-single}, one sees that for $n\geq N$,
 the problem
\begin{equation}\label{u-n}
\left\{
\begin{aligned}
&u_t = d\left[\int_{z_i -n}^{z_i +n}J(x-y)u(t,y)dy-u(t,x)\right] + f(u),&  &  t>0,~x\in [z_i-n ,z_i +n],\\
& u(0,x)=u_0(x),&  & x\in[z_i-n ,z_i +n]
\end{aligned}
\right.
\end{equation}
admits a unique positive steady state $\tilde u_{i,n} \in C([z_i-n ,z_i +n])$ and its unique solution $u_{i,n}(t,x)$ converges to $\tilde u_{i,n}(x) $ in $C([z_i-n ,z_i +n])$ as $t\rightarrow +\infty$, $i=1,2$. Moreover,  if $m>n$, then $u_{i,m}(t,x)\geq  u_{i,n}(t,x)$ for $t>0, z_i-n \leq x \leq z_i +n$.

By making use of Lemma \ref{lemma-MP2} and the ODE problem
\[
\overline u'=f(\overline u),\; \overline u(0)=\|u_0\|_\infty,
\]
we deduce $u_{i,n}(t,x)\leq \overline u(t)$ and hence
  $\tilde u_{i,n}\leq K_0$.
  Therefore there exists $\tilde u_i \in L^{\infty}(\mathbb R)$ such that $\tilde u_{i,n}(x) $ converges to $\tilde u_i(x)$ for every $x\in\mathbb R$ as $n\rightarrow +\infty$.

Since $f=f(u)$, one sees that $\tilde u_1(x-z_1+z_2) = \tilde u_2(x)$.
We now show that
\begin{equation}\label{claim}
\tilde u_1(x) \equiv \tilde u_2(x).
\end{equation}
Fix $N_1\in\mathbb N$ with $N_1 > |z_1 - z_2|$. Then $[z_1-n-N_1, z_1+n+N_1]\supset [z_2-n, z_2+n]$ and hence, for $n\geq N$,
$$
 u_{1,n+N_1}(t, x)\geq  u_{2,n}(t, x) \mbox{ and so } \tilde u_{1,n+N_1}(x)\geq \tilde u_{2,n}(x) \ \ \ \textrm{for}\ x\in [z_2-n ,z_2 +n],
$$
which implies that $\tilde u_1 \geq \tilde u_2$ in $\mathbb R$ by letting $n\rightarrow +\infty$.
Similarly, we have $\tilde u_1 \leq \tilde u_2$ in $\mathbb R$. Hence \eqref{claim} holds.

 From \eqref{claim}  it follows immediately that $\tilde u_1(x)  \equiv c_0$ is a constant in $\mathbb R$ as $z_1$ and $z_2$ are arbitrary.
 It then follows that $\tilde u_{i,n}(x) $ converges to $c_0$ locally uniformly  in $\mathbb R$ as $n\rightarrow +\infty$ thanks to Dini's theorem.
 This in turn implies that $c_0$ is  a steady state of \eqref{fixedbdry} with $(\ell_1,\ell_2)$ is replaced by $(-\infty, +\infty)$
 and hence must be a positive zero of $f(u)$. Thus $c_0=v_0$ and $\tilde u_{1,n}(x)= u_{(z_1-n, z_1+n)}(x) $ converges to $v_0$ locally uniformly  in $\mathbb R$ as $n\rightarrow +\infty$, which implies \eqref{lim-u}.
\end{proof}

\subsection{The spreading-vanishing dichotomy and criteria}
\noindent

Throughout this subsection, we always assume that {\bf (J)} and
{\bf (f1)--(f4)} hold. Let $u_0$ satisfy \eqref{102}, and $(u,g,h)$ be the unique solution of \eqref{101}.
 Since $h(t)$ and $-g(t)$ are increasing in
time $t$,  $h_\infty:=\lim_{t\rightarrow+\infty}h(t)\in(h_0,
+ \infty]$ and $g_\infty:=\lim_{t\rightarrow+\infty}g(t)\in[
-\infty,-h_0)$ are well-defined. Theorems \ref{thm2} and \ref{thm3} will follow from the results proved below.

\begin{theorem}\label{thm-vanishing}
If $h_\infty-g_\infty< + \infty$,
then $u(t,x)\rightarrow0$ uniformly in $[g(t),h(t)]$ as
$t\rightarrow+\infty$ and
$\lambda_p(\mathcal{L}_{(g_\infty, h_\infty)}+ f'(0))\leq 0$.
\end{theorem}

\begin{proof}
We first prove that
$$
\lambda_p(\mathcal{L}_{(g_\infty, h_\infty)}+ f'(0))\leq 0.
$$
Suppose that
$
\lambda_p(\mathcal{L}_{(g_\infty, h_\infty)}+ f'(0)) > 0$.
Then
$\lambda_p(\mathcal{L}_{(g_\infty + \epsilon, h_\infty - \epsilon)}+ f'(0)) > 0$ for small $\epsilon>0$, say $\epsilon\in (0,\epsilon_1)$. Moreover, for such $\epsilon$, there exists $ T_{\epsilon}>0$ such
that
$$
h(t)>h_\infty-\epsilon, \ \ g(t)<g_\infty+\epsilon\;\;\;  \mbox{ for $t> T_{\epsilon}$}.
$$
Consider
\begin{equation}\label{single-epsilon}
\left\{
\begin{aligned}
&w_t=d\int_{g_\infty + \epsilon}^{h_\infty - \epsilon}J(x-y)w(t,y)dy-dw +f(w),
&& t >  T_{\epsilon},~x\in [g_\infty + \epsilon, h_\infty - \epsilon],\\
&w( T_{\epsilon} ,x)=u(T_{\epsilon}, x), && x\in [g_\infty + \epsilon, h_\infty - \epsilon].
\end{aligned}
\right.
\end{equation}
Since $\lambda_p(\mathcal{L}_{(g_\infty + \epsilon, h_\infty - \epsilon)}+ f'(0)) > 0$, Proposition \ref{prop-single} indicates that
the  solution   $ w_{\epsilon}(t,x)$ of   (\ref{single-epsilon})  converges to the unique steady state $W_{\epsilon}(x)$ of (\ref{single-epsilon}) uniformly in $[g_\infty + \epsilon, h_\infty - \epsilon]$ as $t \rightarrow+\infty$.

Moreover, by Lemma \ref{lemma-MP2} and a simple  comparison argument we have
$$
u(t,x)\ge  w_{\epsilon}(t,x)~\text{ for }~t> T_{\epsilon}~\text{ and }~x\in[g_\infty+\epsilon,h_\infty-\epsilon].
$$
Thus, there exists $T_{1\epsilon}> T_{\epsilon}$ such that
$$
u(t,x)\ge {1\over 2} W_{\epsilon}(x)>0~\text{ for }~t> T_{1\epsilon}~\text{ and }~x\in[g_\infty+\epsilon,h_\infty-\epsilon].
$$

Note that since $J(0)>0$, there exist $\epsilon_0>0$ and $\delta_0>0$  such that $J(x)> \delta_0$ if $|x| < \epsilon_0$. Thus for $0<\epsilon <
\min\big\{\epsilon_1, \epsilon_0 /2\big\}$ and $t> T_{1\epsilon}$, we have
\begin{eqnarray*}
h'(t) &= & \mu\int_{g(t)}^{h(t)}\int_{h(t)}^{+\infty} J(x-y)u(t,x)dydx
      \geq  \mu\int_{g_\infty+\epsilon}^{h_\infty - \epsilon}\int_{h_\infty}^{+\infty} J(x-y)u(t,x)dydx\\
      &\geq & \mu\int_{h_\infty - \epsilon_0 /2}^{h_\infty - \epsilon}\int_{h_\infty}^{h_\infty + \epsilon_0 /2} \delta_0 {1\over 2} W_{\epsilon}(x) dydx >0.
\end{eqnarray*}
 This implies  $h_\infty=+\infty$, a contradiction  to the assumption that $h_\infty-g_\infty<+\infty$.
Therefore, we must have
$$
\lambda_p(\mathcal{L}_{(g_\infty, h_\infty)}+ f'(0))\leq 0.
$$

We are now ready to show that $u(t,x)\rightarrow 0$ uniformly in $[g(t),h(t)]$ as
$t\rightarrow+\infty$.
Let $\bar u(t,x)$ denote the unique solution of
\begin{equation}
\label{single-infinity}
\left\{
\begin{aligned}
&\bar u_t=d\int_{g_\infty}^{h_\infty}J(x-y)\bar u(t,y)dy-d\bar u(t,x)+f(\bar u),&  & t>0,~x\in [g_\infty,h_\infty],\\
&\bar u(0,x)=\tilde u_0(x), & &  x\in  [g_\infty,h_\infty],
\end{aligned}
\right.
\end{equation}
where
$$
\tilde u_0(x)=u_0(x)~\text{ if }-h_0\le x\le h_0~\text{ and }
\tilde u_0(x)=0~\text{ if }~x>h_0~\text{ or}~x<-h_0.
$$
By Lemma \ref{lemma-MP-u}, it follows that
$u(t,x)\le\bar u(t,x)$ for $t>0$ and $x\in[g(t),h(t)]$.

Since
$$
\lambda_p(\mathcal{L}_{(g_\infty, h_\infty)}+ f'(0))\leq 0.
$$
Proposition \ref{prop-single} implies that   $\overline u(t,x)
\rightarrow0$ uniformly in $x\in [g_\infty,h_\infty]$ as $t
\rightarrow+\infty$. Hence $u(t,x)\rightarrow0$ uniformly in
$x\in[g(t),h(t)]$ as $t\rightarrow+\infty$. This completes
the proof.
\end{proof}

\begin{lemma}\label{lemma-same}
 $h_\infty< + \infty$ if and only if
$-g_\infty<  + \infty$.
\end{lemma}

\begin{proof}
Arguing indirectly, we assume, without loss of generality,  that $h_\infty=+\infty$
and $-g_\infty<+\infty$.     By Proposition \ref{prop-EV},  there exists  $h_1>0$ such that $\lambda_p(\mathcal{L}_{(0, h_1)}+ f'(0)) > 0$. Moreover, for any $\epsilon>0$ small, there exists $ T_{\epsilon}>0$ such
that
$
h(t)>h_1, \ \ g(t)<g_\infty+\epsilon<0
$ for $t> T_{\epsilon}$.
In particular,
$$
\lambda_p(\mathcal{L}_{(g_\infty+\epsilon, h_1)}+ f'(0)) >\lambda_p(\mathcal{L}_{(0, h_1)}+ f'(0))  > 0.
$$
We now consider
\begin{equation*}
\left\{
\begin{aligned}
&w_t=d\int_{g_\infty + \epsilon}^{h_1}J(x-y)w(t,y)dy-dw +f(w),
& & t >  T_{\epsilon},~x\in [g_\infty + \epsilon, h_1],\\
& w( T_{\epsilon} ,x)=u(T_{\epsilon}, x),  & & x\in [g_\infty + \epsilon, h_1].
\end{aligned}
\right.
\end{equation*}
Similar to the first part of the proof  of Theorem \ref{thm-vanishing},   by choosing $\epsilon < \epsilon_0/2$, we have
$g'(t) <- c<0 $ for all $t$ large. This is a contradiction to  $-g_\infty<+\infty$.
\end{proof}

\begin{theorem}\label{thm-spreading}
If $h_\infty
-g_\infty= + \infty$, then $\lim_{t\rightarrow+\infty}u(t,x)=v_0$ locally uniformly in $\mathbb{R}$, where $v_0$ is the
unique positive zero of the function $f(u)$.
\end{theorem}

\begin{proof}
Thanks to Lemma \ref{lemma-same}, $h_\infty-g_\infty=+\infty$ implies that $h_{\infty} = -g_{\infty}= +\infty$. Choose an increasing sequence $\{ t_n\}_{n\geq 1}$ satisfying
$$
\lim_{n\rightarrow +\infty} t_n =   +\infty,\; \lambda_p(\mathcal{L}_{(g(t_n), h(t_n))}+ f'(0))>0 \mbox{ for all } n\geq 1.
$$

Denote $g_n = g(t_n)$, $h_n = h(t_n)$ and let $\underline u_n(t,x)$  be the unique solution of the following problem
\begin{equation}\label{single-tn}
\left\{
\begin{aligned}
& \underline u_t=d\int_{g_n}^{h_n}J(x-y)\underline u(t,y)dy
-d\underline u(t,x)+f(\underline u),&  & t>t_n,~x\in  [g_n,h_n],\\
& \underline u(t_n,x)= u(t_n,x),& &  x\in  [g_n,h_n].
\end{aligned}
\right.
\end{equation}
By Lemma \ref{lemma-MP2} and a comparison argument we have
\begin{equation}\label{pf-thm-tn>}
u(t,x)\ge\underline u_n(t,x)~\text{ in }~[t_n,+\infty)\times[g_n,h_n].
\end{equation}
Since
$
\lambda_p(\mathcal{L}_{[g_n, h_n]}+ f'(0)) > 0$, by
 Proposition \ref{prop-single}, (\ref{single-tn}) admits a unique positive steady state $\underline u_n(x)$ and
\begin{equation}\label{pf-thm-limit}
\lim_{t\rightarrow+\infty} \underline u_n (t,x)=\underline u_n (x)~\text{ uniformly in } \ [g_n,h_n].
\end{equation}

By Proposition \ref{prop-whole},
\[
\lim_{n\to\infty}  \underline u_n(x)  =v_0 \mbox{ locally uniformly in } x\in\mathbb R.
\]
It follows from this fact,  (\ref{pf-thm-tn>}) and (\ref{pf-thm-limit}) that
\begin{equation}\label{pf-thm-liminf}
\liminf_{t\rightarrow+\infty}u(t,x)\ge v_0~\text{ locally
uniformly in }~\mathbb{R}.
\end{equation}

To complete the proof, it remains to prove that
\begin{equation}\label{pf-thm-limsup}
\limsup_{t\rightarrow+\infty}u(t,x)\le v_0~\text{ locally
uniformly in }~\mathbb{R}.
\end{equation}
Let $\hat u(t)$ be the unique solution of the ODE problem
\[
\hat u'=f(\hat u),\; \hat u(0)=\|u_0\|_\infty.
\]
By Lemma \ref{lemma-MP-u} we have
$u(t,x)\le\hat u(t)$ for $t>0$ and $x\in[g(t),h(t)]$. Since $\hat u(t)\to v_0$ as $t\to\infty$, \eqref{pf-thm-limsup} follows immediately.
\end{proof}

Combining Theorems \ref{thm-vanishing} and \ref{thm-spreading}, we obtain
the following Spreading-Vanishing Dichotomy:

\begin{theorem}[Spreading-vanishing dichotomy]\label{thm-both}
 One of the following alternative must happen for \eqref{101}:
\begin{itemize}
\item[(i)] \underline{\rm Spreading:} $h_\infty=-g_\infty= + \infty$ and $\lim_{t
\rightarrow+\infty}u(t,x)=v_0$ locally uniformly  in
$\mathbb{R}$;
\item[(ii)] \underline{\rm Vanishing:} $h_\infty-g_\infty < + \infty $ and $\lim_{t
\rightarrow+\infty}u(t,x)=0$ uniformly in $[g(t),h(t)]$.
\end{itemize}
\end{theorem}

Next we look for criteria guaranteeing spreading or vanishing for \eqref{101}. From Proposition \ref{prop-EV} we see that if
\begin{equation}\label{d-big}
f'(0)\geq d,
\end{equation}
then $\lambda_p(\mathcal{L}_{(\ell_1,\ell_2)}+ f'(0)) > 0$ for any finite interval $(\ell_1,\ell_2)$. Combining this with Theorem \ref{thm-vanishing} we
immediately obtain the following conclusion:

 \begin{theorem}\label{f'(0)>d}
  When \eqref{d-big} holds,  spreading always happens for \eqref{101}.
\end{theorem}

We next consider the case
\begin{equation}\label{d-small}
f'(0)< d,
\end{equation}
In this case, by Proposition \ref{prop-EV}, there exists $\ell^*>0$ such that
\[
\lambda_p(\mathcal{L}_{I}+ f'(0))=0 \mbox{ if } |I|=\ell^*,\; \lambda_p(\mathcal{L}_I+ f'(0))<0 \mbox{ if } |I|<\ell^*,\; \lambda_p(\mathcal{L}_I+ f'(0))>0 \mbox{ if } |I|>\ell^*,
\]
where $I$ stands for a finite open interval in $\mathbb R$, and $|I|$ denotes its length.

\begin{theorem}\label{thm-mu-small} Suppose that \eqref{d-small} holds.
If $h_0\geq \ell^*/2$ then spreading always happens for \eqref{101}. If $h_0<\ell^*/2$,  then there
exists $\underline\mu>0$ such that vanishing happens for \eqref{101}  if
$0<\mu\le\underline\mu$.
\end{theorem}

\begin{proof} If $h_0\geq \ell^*/2$ and vanishing happens, then $[g_\infty, h_\infty]$ is a finite interval with length strictly bigger than $2h_0\geq \ell^*$.
Therefore $\lambda_p(\mathcal{L}_{(g_\infty, h_\infty)}+ f'(0))>0$, contradicting the conclusion in Theorem \ref{thm-vanishing}. Thus when $h_0\geq \ell^*/2$,
spreading always happens for \eqref{101}.

We now consider the case $ h_0<\ell^*/2$. We fix   $h_1\in (h_0, \ell^*/2)$ and consider the
following problem
\begin{equation}\label{single-h1}
\left\{
\begin{aligned}
&w_t(t,x)=d\int_{-h_1}^{h_1}J(x-y)w(t,y)dy-dw +f(w),
& &t>0,~x\in [- h_1, h_1],\\
&w(0,x)=u_0 (x), & &  x\in [- h_0, h_0], \\
& w(0,x ) = 0,   & &  x\in [- h_1, -h_0) \cup ( h_0, h_1]
\end{aligned}
\right.
\end{equation}
and  denote its unique solution by $\hat w(t,x)$.
The choice of $h_1$ guarantees that
$$
\lambda_1 := \lambda_p(\mathcal{L}_{(-h_1, h_1)}+ f'(0)) < 0.
$$
Let $\phi_1>0$ be the corresponding normalized eigenfunction of $\lambda_1$, namely $\|\phi_1\|_\infty=1$ and
\[
(\mathcal{L}_{(-h_1, h_1)}+ f'(0))[\,\phi_1](x)=\lambda_1\phi_1(x) \mbox{ for } x\in [-h_1, h_1].
\]

By \textbf{(f3)-(f4)}, one has
\begin{eqnarray*}
\hat w_t(t,x) &= & d\int_{-h_1}^{h_1}J(x-y)\hat w(t,y)dy-d\hat w +f(\hat w)\\
        & \leq & d\int_{-h_1}^{h_1}J(x-y)\hat w(t,y)dy-d\hat w +f'(0) \hat w.
\end{eqnarray*}
On the other hand, for $C_1>0$ and $w_1 = C_1 e^{\lambda_1 t/4} \phi _{1}$  it is easy to check that
\begin{eqnarray*}
&&    d\int_{-h_1}^{h_1}J(x-y)w_1(t,y)dy  -  d w_1  + f'(0) w_1 - w_{1t}(t,x)\\
        & = & C_1 e^{\lambda_1 t/4 }\left \{  d\int_{-h_1}^{h_1}J(x-y)  \phi _1 (y)dy-d  \phi _1 +f'(0)  \phi _1  - {\lambda_1\over 4} \phi _1 \right\}\\
        & = &{3\lambda_1\over 4}C_1 e^{\lambda_1 t/4 } \phi _1<0.
\end{eqnarray*}
 Choose
$C_1>0$ large such that $C_1\phi _1> u_0$ in $[-h_1, h_1]$.  Then we can apply Lemma \ref{lemma-MP2} to $w_1-\hat w$ to deduce
\begin{equation}\label{pf-thm-w-decay}
\hat w (t,x)\le w_1 (t,x) = C_1 e^{\lambda_1 t/4} \phi _1 \leq C_1 e^{\lambda_1 t/4} ~\text{ for   }
~t>0~\text{and}~x\in[-h_1, h_1].
\end{equation}

Now define
\begin{equation*}
\hat h(t)=h_0+ 2\mu  h_1  C_1 \int_0^t  e^{\lambda_1 s/4}ds~\text{ and }~\hat g(t)=-\hat
h(t)~\text{ for }~t\ge0,
\end{equation*}
We claim that {\it $(\hat w, \hat h, \hat g)$ is an upper solution of (\ref{101})}.

Firstly, we compute that for any $t>0$,
$$
\hat h(t)=h_0 - 2 \mu h_1 C_1  \frac{4}{\lambda_1 }\left( 1-e^{\lambda_1 t/4} \right) <   h_0 - 2 \mu h_1  C_1  \frac{4}{\lambda_1 }  \leq   h_1
$$
provided that
$$
0< \mu\leq \underline{\mu}:=  \frac{-\lambda_1 (h_1- h_0) }{8h_1  C_1}.
$$
Similarly, $\hat  g(t) > -h_1 $ for any $t>0$. Thus (\ref{single-h1}) gives that
$$
\hat w_t(t,x) \geq d\int_{\hat  g(t)}^{\hat  h(t)}J(x-y)\hat w(t,y)dy-d \hat w +f(\hat w)
 \ \ \  \textrm{for}\  t>0,~x\in [\hat  g(t),\hat  h(t)].
$$
Secondly, due to (\ref{pf-thm-w-decay}),  it is easy to check that
\begin{eqnarray*}
\int_{\hat g(t)}^{\hat h(t)}\int_{\hat h(t)}^{+\infty}J(x-y)\hat w(t,x) dydx  <  2 h_1 C_1 e^{\lambda_1 t/4}.
\end{eqnarray*}
Thus
$$
\hat h' (t)= 2  \mu h_1 C_1    e^{\lambda_1 t/4}  >  \mu \int_{\hat g(t)}^{\hat h(t)}\int_{\hat h(t)}^{+\infty}J(x-y)\hat w(t,x) dydx.
$$
Similarly, one has
$$
\hat g' (t) <-  \mu \int_{\hat g(t)}^{\hat h(t)}\int_{-\infty}^{\hat g(t)}  J(x-y)\hat w(t,x) dydx.
$$

Now it is clear that $(\hat w, \hat h, \hat g)$ is an upper solution of (\ref{101}).
Hence, by Theorem \ref{thm-CP}, we have
$$
u(t,x ) \leq \hat w (t,x),\ g(t) \geq \hat g(t)\ \textrm{and}\  h(t)\leq \hat h(t) \ \ \textrm{for} \ t>0,\ x\in [g(t), h(t)].
$$
Therefore
$$
h_\infty-g_\infty\le\lim_{t\rightarrow+\infty}\left(\hat h(t) -\hat g(t)\right)\leq 2 h_1 < + \infty.
$$
This completes the proof.
\end{proof}

\begin{theorem}\label{thm-mu-large} Suppose that \eqref{d-small} holds and $h_0<\ell^*/2$.
 Then there exists  $\bar\mu > 0$ such
that spreading happens to \eqref{101} if $\mu > \bar\mu$.
\end{theorem}

\begin{proof}
Suppose that  for any $\mu >0$,   $h_\infty-g_\infty<+\infty$. We will derive a contradiction.

First of all, notice that by Theorem \ref{thm-vanishing}, we have $\lambda_p(\mathcal{L}_{(g_\infty, h_\infty)}+ f'(0))\leq 0.$ This indicates that $   h_\infty-g_\infty \leq \ell^* $. To stress the dependence on $\mu$,  let $(u_{\mu}, g_{\mu}, h_{\mu})$ denote the solution of  (\ref{101}).
By Corollary \ref{corollary-mu-increasing}, $u_{\mu},  -g_{\mu},  h_{\mu}$  are increasing in $\mu>0$. Also denote
$$
h_{\mu, \infty} : = \lim_{t\rightarrow +\infty}  h_{\mu}(t), \ \ g_{\mu, \infty} : = \lim_{t\rightarrow +\infty}  g_{\mu}(t).
$$
Obviously, both $h_{\mu, \infty} $ and $- g_{\mu, \infty}$ are increasing in $\mu$. Denote
$$
H_{\infty} := \lim_{\mu \rightarrow +\infty}h_{\mu, \infty},\ \ G_{\infty} := \lim_{\mu \rightarrow +\infty}g_{\mu, \infty}.
$$

Recall   that since $J(0)>0$, there exist $\epsilon_0>0$ and $\delta_0>0$  such that $J(x)> \delta_0$ if $|x| < \epsilon_0$. Then there exist $\mu_1$, $t_1$ such that for $\mu\geq \mu_1$, $t\geq t_1$, we have $h_{\mu} (t)> H_{\infty}-\epsilon_0/4$. Thus it follows that
\begin{eqnarray*}
\mu &=& \left( \int_{t_1}^{+\infty} \int_{g_{\mu}(\tau)}^{h_{\mu}(\tau)}\int_{h_{\mu}(\tau)}^{+\infty}
J(x-y)u_{\mu}(\tau,x)dydxd\tau \right)^{-1} \left[h_{\mu, \infty} - h_{\mu}(t_1) \right]\\
  &\leq &\left( \int_{t_1}^{t_1+1} \int_{g_{\mu_1}(\tau)}^{h_{\mu_1}(\tau)}\int_{h_{\mu_1}(\tau)+\epsilon_0/4}^{+\infty}
J(x-y)u_{\mu_1}(\tau,x)dydxd\tau \right)^{-1} \ell^*\\
&\leq &\left(\delta_0 \int_{t_1}^{t_1+1} \int_{h_{\mu_1}(\tau) - \epsilon_0/2}^{h_{\mu_1}(\tau)}\int_{h_{\mu_1}(\tau)+\epsilon_0/4}^{h_{\mu_1}(\tau)+\epsilon_0/2}
 u_{\mu_1}(\tau,x)dydxd\tau \right)^{-1} \ell^*\\
   &=  &   \left({1\over 4} \delta_0\epsilon_0\int_{t_1}^{t_1+1} \int_{h_{\mu_1}(\tau) - \epsilon_0/2}^{h_{\mu_1}(\tau)}
  u_{\mu_1}(\tau,x) dxd\tau \right)^{-1} \ell^*< +\infty,
\end{eqnarray*}
which clearly is a contradiction.

The above argument also shows that we can take
$$
\bar\mu :=\max\left\{1, \left({1\over 4} \delta_0\epsilon_0\int_{t_1}^{t_1+1} \int_{h_{\mu_1}(\tau) - \epsilon_0/2}^{h_{\mu_1}(\tau)}
  u_{\mu_1}(\tau,x) dxd\tau \right)^{-1} \ell^*\right\}.
$$
\end{proof}

Below is a sharp criteria in terms of $\mu$ for the spreading-vanishing dichotomy.

\begin{theorem}\label{thm-mu-critical} Suppose that \eqref{d-small} holds and $h_0<\ell^*/2$.
 Then there
exists $\mu^*>0$ such that  vanishing   happens for \eqref{101} if $0<\mu\le\mu^*$
and  spreading  happens for \eqref{101} if $\mu>\mu^*$.
\end{theorem}

\begin{proof}
Define
$$
\Sigma=\left\{\mu:~\mu>0~\text{such that}~h_\infty -g_\infty  < + \infty\right\}.
$$

 By Theorems \ref{thm-mu-small} and \ref{thm-mu-large} we see that
  $0< \sup \ \Sigma < +\infty$.
Again we let $(u_{\mu}, g_{\mu}, h_{\mu})$ denote the solution of  (\ref{101}), and set $h_{\mu,\infty}:=\lim_{t\rightarrow+\infty}h_{\mu}(t)$, $g_{\mu,\infty}:=\lim_{t\rightarrow+\infty}g_{\mu}(t)$, and
denote $ \mu^*  =\sup  \Sigma$.

According to  Corollary \ref{corollary-mu-increasing}, $u_{\mu},  -g_{\mu},  h_{\mu}$  are increasing in $\mu>0$. This immediately gives that if $\mu_1 \in \Sigma$, then $\mu  \in \Sigma$ for any $\mu<\mu_1$ and if $\mu_1 \not\in \Sigma$, then $\mu  \not\in \Sigma$ for any $\mu > \mu_1$. Hence it follows that
\begin{equation}\label{pf-thm-mu-separate}
(0, \mu^*) \subseteq \Sigma,\ \ ( \mu^*, +\infty) \cap \Sigma=\emptyset.
\end{equation}

To complete the proof, it remains to show  that $\mu^* \in  \Sigma$. Suppose that $\mu^* \not\in  \Sigma$. Then $h_{\mu^*,\infty}=  -g_{\mu^*,\infty} =+\infty$. Thus there exists $T>0$ such that $-g_{\mu^*}(t)> \ell^*,  h_{\mu^*}(t) > \ell^*$ for $t\geq T$. Hence  there exists $\epsilon >0$ such that for $| \mu - \mu^* |  <\epsilon$,
$-g_{ \mu}(T)> \ell^*/2  ,  h_{ \mu}(T) > \ell^*/2  $, which implies $\mu\not\in\Sigma$. This clearly contradicts (\ref{pf-thm-mu-separate}).
Therefore $ \mu^* \in  \Sigma$.
\end{proof}

\section*{Acknowledgments}
\noindent

The second author was partially supported by the Australian Research Council. The third author was partially supported by NSF  of China (11431005).
The fourth author was partially supported by NSF of China (11671180, 11731005).
J.-F. Cao would like to thank the China Scholarship Council (201606180067) for financial support during the period
of the overseas study and to express her gratitude to
the School of Science and Technology, University of New
England, for its kind hospitality.

\end{document}